\documentclass[12pt]{amsart}
\usepackage{fullpage}
\usepackage[all]{xy}
\usepackage{stmaryrd,url,amssymb,combelow}
\usepackage[T1]{fontenc}
\newtheorem{theorem}{Theorem}[section]
\newtheorem{lemma}[theorem]{Lemma}
\newtheorem{cor}[theorem]{Corollary}
\newtheorem{conj}[theorem]{Conjecture}
\newtheorem{prop}[theorem]{Proposition}
\theoremstyle{definition}
\newtheorem{defn}[theorem]{Definition}
\newtheorem{hypothesis}[theorem]{Hypothesis}
\newtheorem{example}[theorem]{Example}
\newtheorem{remark}[theorem]{Remark}

\newtheorem{notation}[theorem]{Notation}
\numberwithin{equation}{theorem}

\newcommand{\bA}{\mathbf{A}}
\newcommand{\bB}{\mathbf{B}}
\newcommand{\bE}{\mathbf{E}}
\newcommand{\be}{\mathbf{e}}
\newcommand{\bv}{\mathbf{v}}
\newcommand{\bw}{\mathbf{w}}
\newcommand{\CC}{\mathbb{C}}
\newcommand{\FF}{\mathbb{F}}
\newcommand{\QQ}{\mathbb{Q}}
\newcommand{\ZZ}{\mathbb{Z}}
\newcommand{\calE}{\mathcal{E}}

\newcommand{\calH}{\mathcal{H}}
\newcommand{\calM}{\mathcal{M}}
\newcommand{\calR}{\mathcal{R}}
\newcommand{\calS}{\mathcal{S}}
\newcommand{\frakm}{\mathfrak{m}}
\newcommand{\frako}{\mathfrak{o}}
\newcommand{\llangle}{\langle \! \langle}
\newcommand{\rrangle}{\rangle \! \rangle}
\newcommand{\llbrace}{\{\!\{}
\newcommand{\rrbrace}{\}\!\}}

\DeclareMathOperator{\aff}{aff}
\DeclareMathOperator{\an}{an}
\DeclareMathOperator{\bd}{bd}

\DeclareMathOperator{\fin}{fin}
\DeclareMathOperator{\Frac}{Frac}
\DeclareMathOperator{\Gal}{Gal}
\DeclareMathOperator{\GL}{GL}
\DeclareMathOperator{\inte}{int}
\DeclareMathOperator{\Mod}{\mathbf{Mod}}
\DeclareMathOperator{\Norm}{Norm}

\DeclareMathOperator{\rig}{rig}
\DeclareMathOperator{\sSp}{sSp}
\DeclareMathOperator{\Spec}{Spec}

\begin{document}

\title{Frobenius modules over multivariate Robba rings}
\author{Kiran S. Kedlaya}
\address{Department of Mathematics, University of California, San Diego, La Jolla, CA 92093, USA}
\email{kedlaya@ucsd.edu}
\date{August 28, 2020}

\begin{abstract}
We describe a class of multivariate series rings generalizing the usual Robba ring over a $p$-adic field, and give a basic development of $\varphi$-modules over such rings.
This makes it possible to give a unified survey of a number of recent developments in $p$-adic Hodge theory.
\end{abstract}

\thanks{Thanks to Laurent Berger, Hui Gao, Ruochuan Liu, Gergely Z\'abr\'adi, and Sarah Zerbes for feedback. The author was supported by NSF (grants DMS-1101343, DMS-1501214, DMS-1802161); UCSD (Warschawski Professorship); and IAS (Visiting Professorship 2018--2019).}

\maketitle

In $p$-adic Hodge theory, a prominent role is played by certain $p$-adic power series rings known as \emph{Robba rings}. The Robba ring with coefficients in a nonarchimedean field $K$ is the ring of formal sums
$\sum_{n \in \ZZ} c_n \pi^n$ with coefficients in $K$ which converge on some annulus of the form $* < |\pi| < 1$. In practice, such rings typically come equipped with extra structure, especially a \emph{Frobenius endomorphism} which on the subring of series with integral coefficients is a lift of some power of the Frobenius endomorphism modulo $p$.

In fact, there are several different (but related) constructions in $p$-adic Hodge theory that naturally give rise to Robba rings; these constructions are associated 
via the \emph{field of norms} construction of Fontaine--Wintenberger \cite{fontaine-wintenberger}
to various infinite towers of field extensions over a finite extension of $\QQ_p$. The oldest and most frequently encountered case is that of the $p$-cyclotomic tower, as originally considered
by Fontaine \cite{fontaine-phigamma} in his original theory of $(\varphi, \Gamma)$-modules.
Fontaine's original construction does not involve the Robba ring, but rather the ring obtained from it by restricting to series with bounded coefficients (the \emph{bounded Robba ring}) and then completing for the $p$-adic topology. Subsequently, Cherbonnier--Colmez \cite{cherbonnier-colmez} obtained a descent of Fontaine's construction to the bounded Robba ring, and Berger \cite{berger-inv}
(in conjunction with work of this author \cite{kedlaya-local}) showed that the base extension to the full Robba ring allowed for a direct recovery of Fontaine's functors on Galois representations from the associated $(\varphi, \Gamma)$-modules. This recovery extends to other constructions of interest in (cyclotomic) Iwasawa theory, such as the Bloch--Kato exponential map \cite{berger-bk}.

For various reasons, it is natural to ask to what extent a parallel narrative can be constructed with the cyclotomic tower replaced by other towers. (One reason is to extend explicit formulas in cyclotomic Iwasawa theory to Iwasawa theory for other $p$-adic Lie towers; another is to look for candidate constructions for a $p$-adic analytic Langlands correspondence, more on which below.) However, there are only a few candidate towers that give rise to univariate power series rings. One class is the non-Galois Kummer towers considered by
Breuil \cite{breuil} in the context of $p$-divisible group and Kisin \cite{kisin-crys} in the context of crystalline representations; the full analogue of Fontaine's construction for such towers was introduced by
Caruso \cite{caruso}, while the analogue of Cherbonnier--Colmez descent has been established by Gao--Liu \cite{gao-liu} and in a more robust fashion by Gao--Poyeton \cite{gao-poyeton}. Another candidate class of towers are non-cyclotomic $\ZZ_p$-towers, as originally considered by Fourquaux \cite{fourquaux} and (in more detail) Kisin--Ren \cite{kisin-ren}; the analogue of Fontaine's construction for such towers was introduced by Chiarellotto--Esposito \cite{chiarellotto-esposito}, but there is no complete analogue of Cherbonnier--Colmez descent (see for example \cite[Theorem~0.6]{fourquaux-xie}). Nonetheless, one can make some Iwasawa-theoretic constructions in this context;
for example, Berger--Fourquaux \cite{berger-fourquaux} adapt Berger's description of the Bloch--Kato exponential
and the Perrin-Riou map \cite{berger-bk} to this setting.

Unfortunately, there are essentially no other candidate towers which give rise to a \emph{univariate} series ring; a formal statement of this form has been articulated by  Cais--Davis \cite{cais-davis}. In some sense, the difficulty arises from the failure of imperfect fields of characteristic $p$ (notably including power series fields) to admit \emph{canonical} lifts to discrete valuation rings,

The purpose of this paper is to develop a \emph{multivariate} generalization of Robba rings, and Frobenius modules over same; this is meant to encompass a number of recent developments in $p$-adic Hodge theory in which such rings appear. Here is a representative (but not exhaustive) sample.
\begin{itemize}
\item
For the compositum of all $\ZZ_p$-extensions over a finite extension of $\QQ_p$, one may modify the field of norms construction to obtain series in more than one variable; in the case of an unramified extension, this has been described explicitly by Berger \cite{berger-multi} in order to address the failure of overconvergent descent in the Fourquaux--Kisin--Ren setting.
\item
It is natural to extend $p$-adic Hodge theory to allow consideration of Galois representations valued in affinoid algebras; this gives rise to series rings formed from the Robba ring by adjoining new variables which are ``inert'', in the sense that one extends
endomorphisms of the Robba ring so that they fix the new variables. 
These have been considered by Berger--Colmez \cite{berger-colmez}, Kedlaya--Pottharst--Xiao
\cite{kpx}, Hellmann--Schraen \cite{hellmann-schraen}, and others.
\item
In the other direction, it is also natural to consider representations of \'etale fundamental groups of rigid analytic spaces; this gives rise to series rings formed from the Robba ring by adjoining new variables which are not inert, in that one considers some endomorphisms that do not fix these variables.
These have been considered by Andreatta--Brinon \cite{andreatta-brinon}, Kedlaya--Liu \cite{kedlaya-liu2}, and others; these constructions admit strong links with the theory of \emph{perfectoid spaces} in the sense of Scholze \cite{scholze1, scholze2}.
\item
Colmez has described a $p$-adic Langlands correspondence for $\GL_2(\QQ_p)$
in terms of classical $(\varphi, \Gamma)$-modules \cite{colmez-rep}; this involves interpreting the Robba ring in terms of the group algebra of a certain $p$-adic Lie group.
Motivated by this, various authors have proposed analogous constructions starting from other groups, notably Schneider--Vigneras--Z\'abr\'adi \cite{svz, zabradi-robba}.
\item
Following up on the previous example, Z\'abr\'adi \cite{zabradi-phigamma2} has established a correspondence between certain multivariate $(\varphi, \Gamma)$-modules and representations of \emph{products} of Galois groups of $p$-adic fields; these again admit overconvergent descent, as recently shown by Pal--Z\'abr\'adi \cite{pal-zabradi}. This construction can also be interpreted in terms of Drinfeld's lemma for rigid analytic spaces, as in the work of Carter--Kedlaya--Z\'abr\'adi \cite{carter-kedlaya-zabradi}.
\end{itemize}

In order to construct a framework that can encompass such disparate examples, we work initially under rather relaxed hypotheses: our Robba rings arise from lifting affinoid algebras (or even \emph{semiaffinoid algebras} in the sense of Kappen \cite{kappen}) from characteristic $p$,
and our ``Frobenius endomorphisms'' lift certain endomorphisms in characteristic $p$ which need not coincide with a power of absolute Frobenius. However, not all aspects of the usual theory of Frobenius modules over Robba rings remains applicable in such generality, so some of our discussions are forced to take place under more restrictive hypotheses. (One restrictive hypothesis that we insist upon is that our generalized Robba rings are still commutative, although we we will consider noncommutative algebras over these rings in order to describe extra structures. Noncommutative Robba rings associated to group rings have been described in \cite{zabradi-robba}.)

To construct the analogue of a $(\varphi, \Gamma)$-module in our setting, we consider what we call an \emph{enhanced Frobenius lift}; by this, we mean a not necessarily commutative algebra over a particular Robba ring containing a distinguished element $\varphi$ which skew-commutes with the Robba ring via a Frobenius lift. A typical example would be a (suitably completed) group algebra for the product of the free monoid generated by the Frobenius lift with some other group. An \emph{enhanced $\varphi$-module} is then a left module for this ring which is finite projective as a module over the Robba ring. (The restriction to finite projective modules is partly made for expository simplicity, and partly because this is the most crucial case for applications to Galois representations. See \cite{kedlaya-liu2} and especially \cite{kedlaya-liu-finiteness} for some examples of the use of nonprojective modules in this context.)

At this level of generality, one can identify the \emph{\'etale} enhanced $\varphi$-modules, show that they descend canonically to the bounded Robba ring. In the case where the enhanced Frobenius lift is constructed from an action of a $p$-adic Lie group, this descent can be used to identify \emph{locally analytic vectors} for the action of the group on the enhanced $\varphi$-module. One also shows that base extension of \'etale enhanced $\varphi$-modules from the bounded Robba ring to its $p$-adic completion is fully faithful; it is an interesting problem to identify general conditions under which this base extension is also essentially surjective. This amounts to finding analogues
of the theorem of Cherbonnier--Colmez \cite{cherbonnier-colmez} on the overconvergence of classical (i.e., cyclotomic) $(\varphi, \Gamma)$-modules; one such analogue is the aforementioned theorem of Gao--Liu and Gao--Poyeton \cite{gao-liu, gao-poyeton} on overconvergence of $(\varphi, \tau)$-modules in the sense of Caruso.

\begin{notation}
Throughout this paper, let $\frako$ be a complete discrete 
valuation ring with maximal ideal $\frakm$, fraction field $K$, and residue field $k$.
For the moment we do not restrict the characteristic of $K$,
but we will do so later (Hypothesis~\ref{H:Frobenius}).
For convenience only, fix a generator $\varpi$ of $\frakm$ and a normalization of the $\frakm$-adic absolute value on $\frako$;
nothing will depend essentially on these choices.
\end{notation}

\section{Semiaffinoid algebras}
\label{sec:mixed affinoid}

We start by describing a generalization of the usual class of affinoid algebras over a complete discretely valued field, appearing in Kappen's construction of \emph{uniformly rigid spaces} \cite{kappen}. This is needed to optimize the applicability of our general framework; however, many natural examples require only affinoid algebras, so one may safely skip this section on a first reading.

\begin{defn}
For $m$ and $n$ two nonnegative integers, let
\[
\frako \langle T_1,\dots,T_m \rangle \llangle U_1, \dots, U_n \rrangle
\]
denote the $\frakm$-adic completion of $\frako[T_1,\dots,T_m]\llbracket U_1,\dots,U_n \rrbracket$; this ring is noetherian and regular, and even excellent
(see \cite[Proposition~7]{valabrega1} or \cite[Theorem~9]{valabrega2} depending on whether
$\frako$ is of equal or mixed characteristic).
Put 
\[
K \langle T_1,\dots,T_m \rangle \llangle U_1, \dots, U_n \rrangle = 
\frako \langle T_1,\dots,T_m \rangle \llangle U_1, \dots, U_n \rrangle \otimes_{\frako} K,
\]
viewed as a linearly topologized ring with $\frako \langle T_1,\dots,T_m \rangle \llangle U_1, \dots, U_n \rrangle$ as a ring of definition and $(\varpi)$ as an ideal of definition.

A \emph{semiaffinoid algebra} over $K$
is a $K$-algebra $A$ which can be written as a quotient of 
$K \langle T_1,\dots,T_m \rangle \llangle U_1, \dots, U_n \rrangle$ for some $m,n$;
such a quotient map is called a \emph{presentation} of $A$,
and the image of $\frako \langle T_1,\dots,T_m \rangle \llangle U_1, \dots, U_n \rrangle$
(equipped with the quotient map) is called an \emph{integral model} of $A$.
(This is shorthand for the terminology of \cite{kappen}, where an integral model 
in our sense is an \emph{affine flat formal model of ff (formally finite) type}.)

The topology induced by a presentation does not depend on the choice of this presentation;
moreover, for this topology, any $K$-linear homomorphism of semiaffinoid algebras is continuous. To see this, use \cite[Corollary~2.14]{kappen} to lift a homomorphism of semiaffinoid algebras to a homomorphism of some integral models, then apply \cite[Lemma~2.1]{kappen}.
\end{defn}

While the Noether normalization theorem for affinoid algebras does not extend to semiaffinoid algebras
(see \cite[\S 2B]{kappen}), nonetheless one can prove the following statement.
\begin{lemma} \label{L:mixed Nullstellensatz}
Let $A$ be a semiaffinoid algebra over $K$. 
\begin{enumerate}
\item[(a)]
Every maximal ideal of $A$ has residue field finite over $K$.
\item[(b)]
The ring $A$ is a Jacobson ring.
\end{enumerate}
\end{lemma}
\begin{proof}
For (a), see \cite[Lemma~2.3]{kappen};
for (b), see \cite[Proposition~2.17]{kappen}.
\end{proof}

\begin{lemma} \label{L:flat completion}
Let $A \to B$ be a homomorphism of noetherian rings. Choose $f \in A$ and let $\widehat{A}, \widehat{B}$ denote the $f$-adic completions of $A,B$. If
$A/f^n A \to B/f^n B$ is flat for each positive integer $n$ (so in particular if $A \to B$ is flat), then $\widehat{A} \to \widehat{B}$ is flat.
\end{lemma}
\begin{proof}
Apply \cite[Tag~0523]{stacks-project} with $R = A, S = B, I = (f), M = S$.
\end{proof}

\begin{defn}
For $R$ a ring of characteristic $p>0$, we define the \emph{core} of $R$
to be the subring of elements of $R$ which are $p^n$-th powers for all positive integers $n$. This subring is evidently preserved by automorphisms of $R$ (i.e., it is a \emph{characteristic subring} of $R$).
\end{defn}

\begin{lemma} \label{L:integral closure of k}
Suppose that $K$ is of characteristic $p>0$ and $k$ is perfect.
Let $A$ be a semiaffinoid algebra over $K$ with core $R$. 
\begin{enumerate}
\item[(a)]
The ring $R$ is a finite $k$-algebra.
\item[(b)]
If $A$ is reduced, then $R$ equals the integral closure of $k$ in $A$.
\item[(c)]
For $N$ the nilradical of $A$, the projection $A \to A/N$ induces an isomorphism of $R$ with the core of $A/N$.
\end{enumerate}
\end{lemma}
\begin{proof}
We prove first that if $A$ is reduced, then the integral closure $R'$ of $k$ in $A$ is a finite $k$-algebra. Since $A$ is noetherian, we may reduce to the case where $A$ is connected. In this case, $R'$ is reduced, connected, and integral over $k$, and hence an algebraic field extension of $k$. In particular, $R'$ injects into the residue field of every maximal ideal of $A$. By Lemma~\ref{L:mixed Nullstellensatz}, every such residue field is finite over $K$, so we may reduce to the case where $A$ is a finite field extension of $K$.
In this case, the residue field $k'$ of $A$ is finite over $k$ and hence also perfect, and the Cohen structure theorem lets us write $A$ as a Laurent series field over $k'$, in which $k'$ is integrally closed. This proves the claim.

We next verify (b). We may again assume that $A$ is connected; by injecting $A$ into the product of the quotients by its minimal prime ideals, we may also assume that $A$ is integral.
Given $x\in R$, for each maximal ideal $I$ of $A$, we may argue as in the previous paragraph that the image of $x$ in $A/I$ is integral over $k$.
By Krull's theorem, $I \cap R = \bigcap_{n=0}^\infty I^{p^n} = \bigcap_{n=0}^\infty I^n = 0$. Consequently, any integral relation for $x$ over $k$ in $A/I$ holds also in $A$, so $x \in R'$.

To conclude, it suffices to check (c). Since $A$ is noetherian, there exists a positive integer $e$ such that $N^{p^e} = 0$. For $x,y \in R$ with $x-y \in N$, we may write $x = u^{p^e}, y = v^{p^e}$ and $x-y = (u-v)^{p^e}$ to see that $u-v$ must also vanish in every residue field of $A$. Therefore $u-v\in N$ and so $x-y = 0$; that is, the map from $R$ to the core of $A/N$ is injective. (This already suffices to imply (a).) Conversely, if $x \in A$ maps to the core of $A/N$, then for each positive integer $n$ there exists $u_n \in A$ such that $x - u_n^{p^n} \in N$. By the previous discussion, the elements 
$u_{n+e}^{p^n}$ all equal a single element $y \in R$ with $x-y \in N$; hence the map from $R$ to the core of $A/N$ is surjective.
\end{proof}

\begin{remark} \label{R:twist structure morphism}
Suppose that $K \cong k((\varpi))$ and that $A$ is a semiaffinoid algebra over $K$ via some morphism $f_1: K \to A$. Let $f_2: K \to A$ be a continuous homomorphism of topological $k$-algebras which is not necessarily $K$-linear with respect to $f_1$. It may not be immediately obvious that $A$ is again a semiaffinoid algebra over $K$ via $f_2$, but this is indeed the case: given a presentation 
$K \langle T_1,\dots,T_m \rangle \llangle U_1, \dots, U_n \rrangle \to A$
which is $K$-linear with respect to $f_1$, we may also form a presentation
\[
K \langle T_1,\dots,T_{m+1} \rangle \llangle U_1, \dots, U_{n+1} \rrangle \to A
\]
which is $K$-linear with respect to $f_2$ by mapping
$T_1,\dots,T_m,U_1,\dots,U_n$ to their images under $f_1$; mapping $U_{n+1}$ to $f_1(\varpi)$; and mapping $T_{m+1}$ to $f_2(\varpi)^h f_1(\varpi^{-1})$ for some positive integer $h$ large enough so that $f_2(\varpi)^h f_1(\varpi^{-1})$ is topologically nilpotent. This also works if $f_2$ is only $k$-semilinear for some automorphism of $k$, rather than $k$-linear.
\end{remark}

\begin{defn} \label{D:semiaffinoid spaces}
The category of \emph{semiaffinoid spaces} over $K$ is defined as the opposite category of semiaffinoid algebras over $K$; for $A$ an affinoid algebra, let $\sSp A$ denote the corresponding semiaffinoid space. Berthelot's generic fiber construction defines a canonical functor from semiaffinoid spaces over $K$ to rigid analytic spaces extending the identity functor on affinoid spaces \cite[\S 2C1]{kappen}. For example, for $A = K \llangle U \rrangle$, the rigid space associated to
$\sSp A$ is the open unit disc over $K$ with coordinate $U$.

As per \cite[Definition~2.22]{kappen}, a \emph{semiaffinoid subdomain} of a semiaffinoid space $\sSp A$ is the semiaffinoid space associated to the ring obtained from $A$
by choosing an integral model; then performing a finite sequence of open immersions, completions, and admissible blowups; and finally tensoring over $\frako$ with $K$.
Using Lemma~\ref{L:flat completion}, one sees that any $K$-linear homomorphism $A \to B$ arising in this manner is flat; by \cite[Corollary~2.25]{kappen}, the map $\sSp B \to \sSp A$
identifies $\sSp B$ with an admissible open subset of $\sSp A$ which is represented by $A \to B$ (i.e., any homomorphism $A \to C$ for which $\sSp C \to \sSp A$ factors through $\sSp B$ factors through $A \to B$). Unlike in rigid analytic geometry, however, it is unclear whether representability is sufficient to characterize semiaffinoid subdomains.

Using semiaffinoid subdomains, one obtains the \emph{uniformly rigid G-topology} on a semiaffinoid space \cite[Definition~2.35]{kappen}; beware that unlike for rigid spaces,
not every finite cover of a semiaffinoid space by semiaffinoid subdomains is an admissible cover. For example, as in \cite[Example~2.42]{kappen}, there is a nonadmissible covering of the closed unit disc by the open unit disc and the unit circle.

For the uniformly rigid G-topology, the structure presheaf is an acyclic sheaf \cite[Theorem~2.41]{kappen}.
One may thus glue to define \emph{uniformly rigid spaces} over $K$
as per \cite[Definition~2.46]{kappen}, together with a comparison functor to rigid spaces
\cite[Proposition~2.55]{kappen}.
\end{defn}

\begin{remark} \label{R:Kiehl}
There is a natural functor from finitely generated modules over a semiaffinoid algebra to coherent sheaves on the associated semiaffinoid space. Using acyclicity of the structure sheaf plus the flatness of the morphisms representing semiaffinoid subdomain,
one sees that this functor is fully faithful \cite[Corollary~2.43]{kappen}.
In the case of an affinoid algebra, Kiehl showed that this functor is also essentially surjective; however, for a general semiaffinoid algebra, this is unknown and quite possibly false
\cite[Conjecture~3.7]{kappen}. 
One does know that the essential image of the functor is closed under formation of submodules and quotients \cite[Corollary~3.6]{kappen}.
\end{remark}

\section{Dagger lifts of semiaffinoid algebras}
\label{sec:dagger lifts} 

We next generalize the observation that for $k$ a perfect field of characteristic $p$, the integral Robba ring over $W(k)$ is an incomplete but henselian discrete valuation ring with residue field $k((\pi))$, by forming similar lifts of semiaffinoid algebras.
(If you skipped \S\ref{sec:mixed affinoid}, read ``affinoid'' for ``semiaffinoid'' hereafter
and take $n=0$ in the following definitions.)
\begin{defn}
For $m,n$ two nonnegative integers
and $r>0$, define the ring 
\[
\bA\langle T_1,\dots,T_m \rangle \llangle U_1, \dots, U_n \rrangle^{r}
\]
as the completion of $\frako[\pi, T_1,\dots,T_m]\llbracket U_1,\dots,U_n \rrbracket[\pi^{-1}]$
for the $r$-Gauss norm
\[ 
\left| \sum_{i,j_*,k_*} a_{i,j_1,\dots,j_m,k_1,\dots,k_n} \pi^i T_1^{j_1} \cdots T_m^{j_m} U_1^{k_1}\dots U_n^{k_n} \right|_r = \max_{i,j_*,k_*} \{\left|a_{i,j_1,\dots,j_m,k_1,\dots,k_n}\right| \left| \varpi \right|^{ri}\}.
\]
Also define
\[
\bA\langle T_1,\dots,T_m \rangle \llangle U_1, \dots, U_n \rrangle ^{\dagger} = 
\bigcup_{r>0} \bA\langle T_1,\dots,T_m \rangle \llangle U_1, \dots, U_n \rrangle^{r}
\]
and identify the quotient of this ring modulo $(\varpi)$ with
$k((\overline{\pi})) \langle \overline{T}_1,\dots,\overline{T}_m \rangle \llangle \overline{U}_1, \dots, \overline{U}_n \rrangle$.
\end{defn}

\begin{lemma} \label{L:noetherian}
For any nonnegative integers $m,n$ and any $r \in \QQ_{>0}$, the rings 
\[
\bA\langle T_1,\dots,T_m \rangle\llangle U_1, \dots, U_n \rrangle^r, \qquad
\bA\langle T_1,\dots,T_m \rangle\llangle U_1, \dots, U_n \rrangle^{\dagger}
\]
are noetherian and regular. (They should also be excellent, but we do not verify this here.)
\end{lemma}
\begin{proof}
Let $R$ be the $(\frakm, \pi)$-adic completion of $\frako[\pi, T_1,\dots,T_m] \llbracket U_1,\dots,U_n \rrbracket$; since $R$ is obtained from a discrete valuation ring by a sequence of polynomial extensions and completions, it is noetherian and regular.
By similar reasoning, $R/\frakm R$ is noetherian and regular.

For $r = s/t \in \QQ_{>0}$, the ring
$\bA\langle T_1,\dots,T_m \rangle\llangle U_1, \dots, U_n \rrangle^r$
may be constructed by taking the completion of 
$R[\varpi^s \pi^{-t}]$ with respect to the ideal $(\varpi)$, then inverting $\pi$; it is thus again noetherian and regular.

The ring $S = \bA\langle T_1,\dots,T_m \rangle\llangle U_1, \dots, U_n \rrangle^{\dagger}$
is a weakly complete finitely generated algebra
over $R$ for the ideal $\frakm R$ and the generator
$\pi^{-1}$; it is thus noetherian by a theorem of Fulton \cite{fulton}. 
Since $\frakm S$ is a principal ideal contained in the Jacobson radical of $S$
(see Remark~\ref{R:Jacobson radical})
and 
\[
S/\frakm S \cong k((\overline{\pi})) \langle \overline{T}_1,\dots,\overline{T}_m \rangle \llangle \overline{U}_1, \dots, \overline{U}_n \rrangle
\]
is known to be regular, it follows that $S$ is regular.
\end{proof}

\begin{defn} \label{D:mixed affinoid}
Let $A$ be a semiaffinoid algebra over $k((\overline{\pi}))$.
A \emph{dagger lift} of $A$ is a flat $\frako$-algebra $S$ with $S/\frakm S \cong A$ admitting a surjection
\[
f: \bA\langle T_1,\dots,T_m \rangle\llangle U_1, \dots, U_n \rrangle^{\dagger} \to S
\]
for some  $m,n$;
any such surjection is called a \emph{presentation} of $S$. By Lemma~\ref{L:noetherian},
$S$ is noetherian.

Given a presentation $f$ of $S$, for $r>0$ define the subrings
\[
S^r = f(\bA\langle T_1,\dots,T_m \rangle\llangle U_1, \dots, U_n \rrangle^{r})
\]
and the quotient norm $\left|\bullet \right|_r$ on $S^r$ induced via $f$.
While these constructions depend on $f$,
any two choices of $f$ give the same rings $S^r$ and equivalent norms $\left| \bullet \right|_r$ for $r$ sufficiently small, by Lemma~\ref{L:dagger compatibility} below.
(Beware that this uniqueness assertion depends not only on $A$ as a topological ring, but also as a $k((\overline{\pi}))$-algebra. See Definition~\ref{D:dagger morphism}.)
\end{defn}

\begin{remark} \label{R:Hadamard}
Let $S$ be a dagger lift of a semiaffinoid algebra over $k((\overline{\pi}))$,
and fix a presentation of $S$.
We will use frequently the fact that for $0 < s \leq r$ and $x \in S^r$,
\[
\left| x \right|_s \leq \left| x \right|_r^{s/r}.
\]
A slightly stronger statement is that for $t \in [0,1]$,
\[
\left| x \right|_{r^ts^{1-t}} \leq \left| x \right|_r^{t} \left| x \right|_s^{1-t}.
\]
Both theorems are proved by reducing to the case $S = \bA\langle T_1,\dots,T_m \rangle\llangle U_1, \dots, U_n \rrangle^\dagger$,
then reducing to the case $x = \pi^i T_1^{j_1} \cdots T_m^{j_m} U_1^{k_1} \cdots U_n^{k_n}$, for which both inequalities become
equalities.
\end{remark}

\begin{remark} \label{R:Jacobson radical}
Let $S$ be a dagger lift of a semiaffinoid algebra over $k((\overline{\pi}))$.
For any $x \in 1 + \frakm S$, by Remark~\ref{R:Hadamard}, we have $\left| 1-x \right|_r < 1$ for
$r>0$ sufficiently small, so $x$ is a unit in $S$. That is, $\frakm S$ is contained in the Jacobson radical
of $S$.
\end{remark}

\begin{remark} \label{R:lim norm}
Let $S$ be a dagger lift of a semiaffinoid algebra $A$ over $k((\overline{\pi}))$.
For $x \in S \setminus \frakm S$ mapping to $\overline{x} \in A$, the expression for $\left| x \right|_r$ is dominated by $\left| \overline{x} \right|^r$ for sufficiently small $r$; in particular, $\lim_{r \to 0^+} \left| x \right|_r = 1$.
It follows that for $x \in S$ nonzero, we have $\lim_{r \to 0^+} \left| x \right|_r = \left| \varpi \right|^{m}$ where
$m$ is the largest integer for which $x \in \frakm S$.
\end{remark}

\begin{defn}
Let $S$ be a dagger lift of a semiaffinoid algebra over $k((\overline{\pi}))$.
For each positive integer $n$, the quotient topology on $S/\frakm^n S$ induced by
$\left| \bullet \right|_r$ is independent of $r$ provided that $r$ is sufficiently small.
By taking the inverse limit of these topologies, we obtain the \emph{weak topology} on $S$. It will follow from
Corollary~\ref{C:same norm} that the definition of the weak topology depends only on $S$ and not on the choice of a presentation.
\end{defn}

\begin{defn} \label{D:dagger morphism}
Let $S_1, S_2$ be dagger lifts of the semiaffinoid algebras $A_1, A_2$ over $k((\overline{\pi}))$. A \emph{dagger morphism} $g: S_1 \to S_2$ is a morphism which sends $\varpi$ to a generator of $\varpi S_2$, is
continuous with respect to the weak topologies, and has the property that the induced
(continuous) morphism  $\overline{g}: A_1 \to A_2$ (the \emph{reduction} of $g$)
is $k$-semilinear with respect to some automorphism of $k$
(but not necessarily $k$-linear, let alone $k((\overline{\pi}))$-linear).

We say that $g$ is \emph{norm-scaling} if for some (hence any) quotient norm on $A_2$ induced by a presentation, the restriction of this norm to $k((\overline{\pi}))$ via $\overline{g}$ is equivalent to a multiplicative norm. This holds in particular if $g$ is $k((\overline{\pi}))$-linear.
\end{defn}

\begin{remark}
An illustrative example of a dagger morphism which is not norm-scaling is the map
\begin{gather*}
g: \bA^\dagger \to\bA \langle T_1, T_2 \rangle^\dagger/(T_1 T_2 - \pi), \\
\pi \mapsto \pi T_1, \quad \pi^{-1} \mapsto \pi^{-2} T_2.
\end{gather*}
A variant is the map
\begin{gather*}
g: \bA\langle T_1, T_2 \rangle^\dagger/(T_1 T_2^2 - \pi)    \to\bA \langle T_1, T_2 \rangle^\dagger/(T_1 T_2 - \pi), \\
\pi \mapsto \pi T_1,\quad \pi^{-1} \mapsto \pi^{-2} T_2, \quad
T_1 \mapsto T_2, \quad
T_2 \mapsto T_1,
\end{gather*}
which is an isomorphism, the inverse map being
\begin{gather*}
h: \bA\langle T_1, T_2 \rangle^\dagger/(T_1 T_2 - \pi)    \to\bA \langle T_1, T_2 \rangle^\dagger/(T_1 T_2^2 - \pi), \\
\pi \mapsto T_1 T_2, \quad \pi^{-1} \mapsto \pi^{-1} T_2, \quad
T_1 \mapsto T_2, \quad
T_2 \mapsto T_1
\end{gather*}
(note that $\pi^2 \mapsto \pi T_1$).
\end{remark}

\begin{remark} \label{R:twist structure morphism2}
Let $S$ be a dagger lift of the semiaffinoid algebra $A$ over $k((\overline{\pi}))$,
and choose a presentation $f: \bA^\dagger\langle T_1, \dots, T_m \rangle \llangle U_1,\dots,U_n \rrangle\to S$.
Let $g: \bA^\dagger \to S$ be a dagger morphism with reduction $\overline{g}$.
By Remark~\ref{R:twist structure morphism}, $A$ also has the structure of a semiaffinoid algebra over $k((\overline{\pi}))$ via $\overline{g}$. For this structure, $S$ may be viewed as a dagger lift of $A$ using the presentation
\[
\bA^\dagger\langle T_1, \dots, T_{m+1} \rangle \llangle U_1,\dots,U_{n+1} \rrangle\to S
\]
with $\bA^\dagger$ mapping via $g$, $T_1,\dots,T_m,U_1,\dots,U_n$ mapping as before,
$U_{n+1}$ mapping to $f(\pi)$, and $T_{m+1}$ mapping to $g(\pi)^h f(\pi^{-1})$ for some positive integer $h$ large enough that $\overline{g}(\overline{\pi})^h \overline{f}(\overline{\pi}^{-1})$ is topologically nilpotent in $A$.
\end{remark}

\begin{lemma} \label{L:dagger compatibility}
For $i=1,2$ and $r>0$, set notation as follows:
\begin{itemize}
\item
let $A_i$ be a semiaffinoid algebra over $k((\overline{\pi}))$; 
\item
let $S_i$ be a dagger lift of $A$;
\item
let $
f_i: \bA\langle T_1,\dots,T_{m_i} \rangle \llangle U_1, \dots, U_{n_i} \rrangle^{\dagger} \to S_i$ be a presentation;
\item
put $S_i^r = f_i(\bA\langle T_1,\dots,T_{m_i} \rangle\llangle U_1, \dots, U_{n_i} \rrangle^{r})$;
\item
let
$\left| \bullet \right|_{i,r}$ be the quotient norm on $S_{i}^r$ induced from $\left| \bullet \right|_r$
via $f_i$; 
\item
let $\overline{f}_i: k((\overline{\pi})) \langle \overline{T}_1,\dots,\overline{T}_{m_i} \rangle \llangle \overline{U}_1, \dots, \overline{U}_{n_i} \rrangle
 \to A_i$
be the reduction of $f_i$ modulo $\frakm$;
\item
let $\left| \bullet \right|_i$ be the quotient norm on $A_i$ induced via $\overline{f}_i$.
\end{itemize}
Let $g: S_1 \to S_2$ be a norm-scaling dagger morphism.
Then there is a unique choice of $t$ for which the following statements hold.
\begin{enumerate}
\item[(a)]
There exists $c_1 > 0$ such that for all $\overline{x} \in A_1$, 
\[
\left| \overline{g}(\overline{x}) \right|_2 \leq c_1 \left| \overline{x} \right|_1^t.
\]
\item[(b)]
There exist $r_0, c_2 > 0$ such that for $0 < r \leq r_0$,
$g$ maps $S_1^{tr}$ to $S_2^{r}$ 
and 
\[
\left|g(x)\right|_{2,r} \leq c_2^r \left|x\right|_{1,tr} \qquad (x \in S_1^{tr}).
\]
\end{enumerate}
We will say that $g$ is \emph{norm-scaling with parameter $t$}.
\end{lemma}
\begin{proof}
Since the restriction of
$\left| \bullet \right|_2$ to $\overline{g}(k((\overline{\pi})))$ is equivalent to a multiplicative norm, the spectral seminorm of $\overline{g}(\overline{\pi})$ must equal $\left| \overline{\pi} \right|^t$ for some $t>0$.
If we equip $A_2$ with the norm $\left| \bullet \right|_2^{1/t}$, we may then view it as a Banach algebra over $k((\overline{\pi}))$ via $\overline{g}$, in which case $\overline{g}$ becomes a continuous and hence bounded morphism of Banach algebras over $k((\overline{\pi}))$.
This proves (a).

Given (a),
choose some $\epsilon>0$.
Choose $r_1>0$ such that $S_1^{tr_1}$ surjects onto $A_1$ and $S_2^{r_1}$ surjects onto $A_2$. 
Set 
\[
u_i = \begin{cases} \pi & i=0 \\ T_i & i=1,\dots,m_1 \\
U_{i-m_1} & i=m_1+1,\dots,m_1+n_1.
\end{cases}
\]
For $i=0,\dots,m_1+n_1$, if $g(u_i) = 0$, set $v_i = 0$; otherwise, define $v_i$ as follows. Let $h_i$ be the largest nonnegative integer for which
$\varpi^{h_i}$ divides $g(u_i)$ and put $w_i = \varpi^{-h_i} g(u_i) \in S_2^{r_1}$.
Lift $\overline{w}_i \in A_2$ to some 
\[
\overline{v}_i \in k((\overline{\pi})) \langle \overline{T}_1,\dots,\overline{T}_{m_2} \rangle \llangle \overline{U}_1, \dots, \overline{U}_{n_2} \rrangle
\]
with $\left| \overline{v}_i \right| \leq (1+\epsilon)\left| \overline{w}_i \right|_2$. Lift $\overline{v}_i$ to
\[
v'_i \in \bA\langle T_1,\dots,T_{m_2} \rangle \llangle U_1, \dots, U_{n_2} \rrangle^{r_1}
\]
in such a way that every monomial $\overline{\pi}^{h} T_1^{j_1} \cdots T_{m_2}^{j_{m_2}} U_1^{k_1}\cdots U_{n_2}^{k_{n_2}}$ with a zero coefficient in $\overline{v}_i$
is lifted to zero.
Then project $v'_i$ to $w'_i \in S_2^{r_1}$. The image of $w'_i$ in
$A_2$ equals $\overline{w}_i$, so $w'_i \equiv w_i \pmod{\varpi}$. We may thus put $w''_i = \varpi^{-1}(w_i - w'_i) \in S_2^{r_1}$ and lift $w''_i$ to 
\[
v''_i \in \bA\langle T_1,\dots,T_{m_2} \rangle\llangle U_1, \dots, U_{n_2} \rrangle^{r_1}.
\]
Put
\[
v_i = \varpi^{m_i}(v'_i + \varpi v''_i);
\]
then $\varpi^{-m_i} v_i$ projects to $\overline{v}_i$ in $k((\overline{\pi})) \langle \overline{T}_1,\dots,\overline{T}_{m_2} \rangle \llangle \overline{U}_1, \dots, \overline{U}_{n_2} \rrangle$ and to $w_i$ in $S_2$.

Note that the special form of $v'_i$ ensures that for all $r>0$,
\[
\left| v'_i \right|_r = \left| \overline{v}_i \right|^{r}.
\]
On the other hand, by Remark~\ref{R:Hadamard}, for $0 < r \leq r_1$,
\[
\left| v''_i \right|_r \leq \left| v''_i \right|_{r_1}^{r/r_1},
\]
and if $v'_i \neq 0$ then there exists $r_0>0$ such that for $0 < r \leq r_0$,
\[
\left| \varpi \right| \left| v_i'' \right|_{r_1}^{r/r_1} < \left| \overline{v}_i \right|^{r}.
\]
Consequently, for $0 < r \leq r_0$ we have
\begin{align*}
\left| v_i \right|_r &= \left| \varpi \right|^{m_i} \left| \overline{v}_i \right|^r \\ &\leq \left| \varpi \right|^{m_i} (1+\epsilon)^r \left| \overline{w}_i \right|_2^r.
\end{align*}
In case $m_i = 0$, then $\overline{w}_i = \overline{g}(\overline{x})$ and so
\[
\left| v_i \right|_r \leq (1+\epsilon)^r c_1^r \left| \overline{u}_i \right|^{tr} = (1+\epsilon)^r c_1^r \left| u_i \right|_{tr}.
\]
In case $m_i>0$, by choosing $r_0$ small enough
we can still ensure that for $0 < r \leq r_0$,
\[
\left| v_i \right|_r \leq (1+\epsilon)^r c_1^r \left| u_i \right|_{tr}.
\]
Of course this also holds if $v_i = 0$.

Now for general $x \in S_1^{tr}$ with $0 < r \leq r_0$,
choose $\tilde{x} \in \bA\langle T_1,\dots,T_{m_1} \rangle \llangle U_1,\dots,U_{n_1} \rrangle^{\dagger,tr}$
lifting $x$ with $\left| \tilde{x} \right|_{tr} \leq (1+\epsilon) \left| x \right|_{1,tr}$. Write
\[
\tilde{x} = \sum_{h,j_*,k_*} \tilde{x}_{h,j_1,\dots,j_{m_1},k_1,\dots,k_{n_1}} \pi^h T_1^{j_1} \cdots T_{m_1}^{j_{m_1}}
U_1^{k_1}\cdots U_{n_1}^{k_{n_1}};
\]
then $g(x)$ is the image in $S_2^{r}$ of
\[
\sum_{h,j_*,k_*} \tilde{x}_{h,j_1,\dots,j_{m_1},k_1,\dots,k_{n_1}} u_0^h u_1^{j_1} \cdots u_{m_1}^{j_{m_1}}
u_{m_1+1}^{k_1} \cdots u_{m_1+n_1}^{k_{n_1}}
\]
and for $0 < s \leq r$,
\begin{align*}
&\left| \tilde{x}_{h,j_1,\dots,j_{m_1},k_1,\dots,k_{n_1}} v_0^h v_1^{j_1} \cdots v_{m_1}^{j_{m_1}} v_{m_1+1}^{k_1}\cdots v_{m_1+n_1}^{k_{n_1}} \right|_{s} \\
&\leq 
\left| \tilde{x}_{h,j_1,\dots,j_{m_1},k_1,\dots,k_{n_1}} \right| (1+\epsilon)^s c^s \left| \pi^h T_1^{j_1} \cdots T_{m_1}^{j_{m_1}}
U_1^{k_1} \cdots U_{n_1}^{k_{n_1}} \right|_{ts} \\
&\leq (1+\epsilon)^s c_1^s \left| \tilde{x} \right|_{ts} \\
&\leq (1+\epsilon)^{2s} c_1^s \left| x \right|_{1,ts}.
\end{align*}
This yields the desired inequality with $c_2 = (1+\epsilon)^2 c_1$.
\end{proof}

\begin{cor} \label{C:dagger compatibility}
For any dagger morphism $g: S_1 \to S_2$, there exists $t>0$ such that for any presentations of $S_1, S_2$, there exists $r_0>0$ such that  $g(S_1^{tr}) \subseteq S_2^r$ for all $0 < r \leq r_0$.
\end{cor}
\begin{proof}
Using Remark~\ref{R:twist structure morphism2}, we may reduce to the case where $g$ is $\bA^\dagger$-linear, and in particular norm-scaling. We may then invoke Lemma~\ref{L:dagger compatibility} to conclude.
\end{proof}

The following corollary asserts that for a given dagger lift $S$ of a given semiaffinoid algebra $A$, any statements made about the subrings $S_i^r$ for ``sufficiently small $r$'' can be made independently of the choice of a presentation.
\begin{cor} \label{C:same norm}
Let $A$ be a semiaffinoid algebra over $k((\overline{\pi}))$
and let $S = S_1 = S_2$ be a dagger lift of $A$.
For $i=1,2$, 
let
\[
f_i: \bA\langle T_1,\dots,T_{m_i} \rangle\llangle U_1, \dots, U_{n_i} \rrangle^{\dagger} \to S_i
\]
be a presentation,
put 
\[
S_i^{r} = f_i(\bA\langle T_1,\dots,T_{m_i} \rangle\llangle U_1, \dots, U_{n_i} \rrangle^{\dagger,r}),
\]
and let
$\left| \bullet \right|_{i,r}$ be the quotient norm on $S_i^{r}$ induced from $\left| \bullet \right|_r$
via $f_i$.
Then there exist $r_0, c > 0$ such that for $0 < r \leq r_0$, $S_1^{r} = S_2^{r}$ and
\[
\left|x\right|_{1,r} \leq c^r \left|x\right|_{2,r},
\quad
\left|x\right|_{2,r} \leq c^r \left|x\right|_{1,r}
\qquad (x \in S_1^{r}).
\]
\end{cor}
\begin{proof}
Apply Lemma~\ref{L:dagger compatibility} to the identity map $g: S \to S$.
\end{proof}

We may further refine Lemma~\ref{L:dagger compatibility} to lift certain properties of ring homomorphisms.

\begin{lemma} \label{L:lift properties}
Let $g: S_1 \to S_2$ be a dagger morphism with reduction $\overline{g}$.
If $\overline{g}$ has one of the following properties:
\begin{enumerate}
\item[(a)]
injective;
\item[(b)]
surjective;
\item[(c)]
finite;
\item[(d)]
flat;
\item[(e)]
faithfully flat;
\item[(f)]
finite \'etale;
\end{enumerate}
then so does $g$. Moreover, if $g$ is norm-scaling with parameter $t$, then in cases (a)--(c) and (f),
for $r>0$ sufficiently small, the morphisms $S_1^{tr} \to S_2^r$ also have the claimed property.
\end{lemma}
\begin{proof}
By Remark~\ref{R:twist structure morphism2}, we may reduce to the case where $g$ is norm-scaling.
Retain notation as in Lemma~\ref{L:dagger compatibility}.
Suppose that $\overline{g}$ is injective. Then $\ker(g) \subseteq \bigcap_{n=1}^\infty \frakm^n S_1$,
but the latter equals 0 because $\frakm S_1$ is contained in the Jacobson radical of $S_1$
(Remark~\ref{R:Jacobson radical}). This proves (a).

Suppose that $\overline{g}$ is finite; we prove that $g$ is finite by a variant of the
proof of Lemma~\ref{L:dagger compatibility}.
Fix presentations
\[
f_i: \bA\langle T_1,\dots,T_{m_i} \rangle\llangle U_1, \dots, U_{n_i} \rrangle^{\dagger} \to S_i
\]
and use these to define subrings $S_i^{r}$ of $S_i$,
quotient norms $\left| \bullet \right|_{i,r}$ on $S_i^{r}$,
and quotient norms $\left| \bullet \right|_i$ on $A_i$.
Choose $r_0 > 0$ for which $S_2^{r_0}$ surjects onto $A_2$
and the conclusion of Lemma~\ref{L:dagger compatibility} holds for some $c>0$.
We can then find finitely many elements $w_1,\dots,w_m \in S_2^{r_0}$ whose images in $A_2$
generate $A_2$ as an $A_1$-module.
By the open mapping theorem 
\cite[\S I.3.3, Th\'eor\`eme~1]{bourbaki-evt}, 
the map $A_1^m \to A_2$ given by $(\overline{y}_1,\dots,\overline{y}_m) \mapsto
\overline{g}(\overline{y}_1) \overline{w}_1 + \cdots + \overline{g}(\overline{y}_m) \overline{w}_m$ is strict;
that is, 
there exists $c_1>0$ such that
for any $\overline{x} \in A_2$, there exist $\overline{y}_1,\dots,\overline{y}_m \in A_1$ with $\overline{g}(\overline{y}_1) \overline{w}_1 + \cdots + \overline{g}(\overline{y}_m) \overline{w}_m = \overline{x}$
and $\left| \overline{y}_1 \right|^t_1, \dots, \left| \overline{y}_m \right|^t_1 \leq c_1 \left| \overline{x} \right|_2$.

Fix $\epsilon > 0$. Given $x \in S_2^{r_0}$, we construct elements $y_{n,j}\in S_1^{t r_0}$ for $n=0,1,\dots$
and $j=1,\dots,m$ as follows.
Given $y_{0,j},\dots,y_{n-1,j}$  for which 
\[
x \equiv \sum_{i=0}^{n-1} \sum_{j=1}^m g(\varpi^i y_{i,j}) w_m \pmod{\varpi^n},
\]
put
\[
z = \varpi^{-n}\left(x - \sum_{i=0}^{n-1} \sum_{j=1}^m g(\varpi^i y_{i,j}) w_m\right)
 \in S_2^{r_0},
\]
and choose $y_{n,j} \in S_1^{r_0}$ with $\overline{g}(\overline{y}_1)\overline{w}_1 + \cdots + \overline{g}(\overline{y}_m)\overline{w}_m = \overline{z}$ and
$\left| y_{n,j} \right|_{1,r_0} \leq c_1^{r_0} (1+\epsilon)^{t r_0} \left| z \right|_{2,r_0}$.
By induction, we have 
\[
\left| \varpi^n y_{n,j} \right|_{1,t r_0} \leq c_1^{(n+1)r_0} (1+\epsilon)^{r_0} c^{r_0}
\max_j\{\left| w_j \right|_{2,r_0}\}.
\]
Put $y_j = \sum_{n=0}^\infty y_{n,j}$.
As in the proof of Lemma~\ref{L:dagger compatibility}, there exist $r_1 \in (0,r_0]$ and $c_2 > 0$
such that for $r \in (0, r_1]$, we have $y_j \in S_1^{r}$ and 
$\left| y_j \right|_{1,tr} \leq c_2^r \left| x \right|_{2,r}$. Since $S_2^{r_0}$ is dense in $S_2^{r_1}$,
it follows that $S_2^{r_1}$ is finite over $S_1^{t r_1}$. This proves (c); the same argument with $m=1$ proves (b) (or see Remark~\ref{R:finite generators}).

Suppose that $\overline{g}$ is flat. To check that $g$ is flat, it suffices to check after localizing
at a maximal ideal $I$ of $S_1$, which by Remark~\ref{R:Jacobson radical} must contain $\frakm S_1$.
Since the rings $S_i$ are noetherian by Lemma~\ref{L:noetherian}, it is enough to check flatness of the maps
of completed local rings, for which we easily deduce flatness from flatness modulo $\frakm$ plus the fact that
$\frakm$ is a principal ideal. This proves (d) and similarly (e).

One obtains (f) from (c) and (d) using the fact that a ring morphism $A \to B$ is finite \'etale if and only if $B$ is finite projective as a module over both $A$ (via the given morphism) and $B \otimes_A B$ (via the multiplication map).
To show that in addition $S_1^{tr} \to S_2^r$ is finite \'etale for $r$ sufficiently small,
note that this morphism is finite for $r$ sufficiently small by (c). Since $g$ is already known to be \'etale,
$S_1^{tr} \to S_2^r$ becomes \'etale after inverting some $x \in S_1^{tr}$; by taking $r$ sufficiently small, 
$x$ becomes invertible and the claim follows.
\end{proof}
\begin{cor} \label{C:lift isomorphism}
Any dagger morphism whose reduction is an isomorphism is itself an isomorphism.
\end{cor}

\begin{remark} \label{R:finite generators}
The proof of Lemma~\ref{L:lift properties}(c) gives a slightly stronger result: any finite collection of elements
of $S_2$ whose images in $A_2$ generate $A_2$ as an $A_1$-module themselves generate $S_2$ as an $S_1$-module.
This also follows directly from the statement of Lemma~\ref{L:lift properties}(c) using
Remark~\ref{R:Jacobson radical} and Nakayama's lemma.

By the same token, if $\overline{g}$ is finite flat, then any presentation of $A_2$ as a direct summand of a finite free $A_1$-module lifts to a presentation of $S_2$ as a direct summand of a finite free module;
and if $\overline{g}$ is finite \'etale, then any $A_1$-linear section of $A_2 \to A_2/A_1$ lifts to an $S_1$-linear section of $S_2 \to S_2/S_1$.
\end{remark}

\begin{remark}
Let $A$ be a semiaffinoid algebra over $k((\overline{\pi}))$
and let $S$ be a dagger lift of $A$. Then each semiaffinoid localization $A \to B$ gives rise to a dagger lift $T$ of $B$ equipped with a morphism $S \to T$; since $A \to B$ is flat, so then is $S \to T$ by Lemma~\ref{L:lift properties}(d).
\end{remark}

\section{Multivariate Robba rings}
\label{sec:multivariate}

\begin{defn} \label{D:Robba ring}
Let $A$ be a semiaffinoid algebra over $k((\overline{\pi}))$, and fix a dagger lift $S$ of $A$ and a presentation of $S$.
For $r>0$, let $\calR_A^r$ be the Fr\'echet completion of $S^r[\varpi^{-1}]$ for the norms
$\left| \bullet \right|_s$ for $s \in (0,r]$. Put $\calR_A = \bigcup_{r>0} \calR_A^r$.
Let $\calR_A^{\inte}$ be the subring of $\calR_A$ consisting of those elements
$x$ for which $\limsup_{r \to 0^+} \left| x \right|_r \leq 1$;
it carries a natural LF (limit of Fr\'echet) topology.
Put $\calR_A^{\bd} = \calR_A^{\inte}[p^{-1}]$.
\end{defn}

\begin{lemma} \label{L:dagger to integral Robba}
In Definition~\ref{D:Robba ring}, the natural map $S \to \calR_A^{\inte}$
is an isomorphism.
\end{lemma}
\begin{proof}
For $r>0$, $\calR^r_A$ is constructed as the Fr\'echet completion of $S^r[\varpi^{-1}]$
with respect to a family of seminorms including the norm $\left| \bullet \right|_r$;
consequently, the map $S^r \to \calR^r_A$ is injective. By taking direct limits, we see that $S \to \calR_A^{\inte}$ is injective. 

To prove that $S \to \calR_A^{\inte}$ is surjective, it suffices to check that an arbitrary nonzero element $x \in \calR_A^{\inte}$ belongs to its image.
(We take this approach because we do not yet have any finiteness property for $\calR_A^{\inte}$.)
Fix a presentation 
$f: \bA\langle T_1,\dots,T_{m} \rangle\llangle U_1, \dots, U_{n} \rrangle^{\dagger} \to S$;
since $\varpi x$ is power-bounded in $\calR_A^{\inte}$, 
we can construct a dagger lift $S'$ of $A$ inside $\calR_A^{\inte}$ containing $\varpi x$  
by extending $f$ to a presentation
$f': \bA\langle T_1,\dots,T_{m},T_{m+1} \rangle\llangle U_1, \dots, U_{n} \rrangle^{\dagger} \to S$
taking $T_{m+1}$ to $\varpi x$.
The map $S \to S'$ is then an isomorphism by Corollary~\ref{C:lift isomorphism}, so $\varpi x$ lifts to an element $y \in S$.
We may now apply Remark~\ref{R:lim norm} to see that $\lim_{r \to 0^+} \left| \varpi x \right|_r = \left| \varpi \right|^m$ where $m$ is the largest integer such that $y \in S$; since $\limsup_{r \to 0^+} \left| x \right|_r \leq 1$ we must have $m \geq 1$, so $x$ is the image of the element $\varpi^{-1} y$ of $S$.
\end{proof}

\begin{lemma} \label{L:units in integral}
We have $\calR_A^{\inte} \cap (\calR_A^{\bd})^\times = \bigcup_{m=0}^\infty \varpi^m (\calR_A^{\inte})^\times$.
\end{lemma}
\begin{proof}
Suppose that $x \in \calR_A^{\inte}$ admits the inverse $y \in \calR_A^{\bd}$. 
Let $m$ be the largest integer such that $x \in \varpi^m \calR_A^{\inte}$;
by Lemma~\ref{L:dagger to integral Robba} plus Remark~\ref{R:lim norm}, 
we have $\lim_{r \to 0^+} \left| \varpi^{-m} x \right|_r = 1$
and hence $\lim_{r \to 0^+} \left| \varpi^{m} y \right|_r = 1$.
It follows that $\varpi^m y \in \calR_A^{\inte}$ and so
$\varpi^{-m} x \in (\calR_A^{\inte})^{\times}$.
\end{proof}

\begin{defn} \label{D:Robba functoriality}
For $i=1,2$, let $A$ be a semiaffinoid algebra over $k((\overline{\pi}))$ 
and let $S_i$ be a dagger lift of $A$.
Let $g: S_1 \to S_2$ be a dagger morphism. 
If $g$ is norm-scaling with parameter $t$ as per Lemma~\ref{L:dagger compatibility}, 
then for $r>0$ sufficiently small the morphisms
$S_1^{tr} \to S_2^r$ induce morphisms $\calR^{tr}_{A_1} \to \calR^r_{A_2}$, and hence a morphism $\calR_{A_1} \to \calR_{A_2}$. The latter conclusion holds also for general $g$
by Remark~\ref{R:twist structure morphism2}.

In particular, if $\overline{g}$ is an isomorphism, then so is $g$ by Corollary~\ref{C:lift isomorphism}. This means that, while the construction of $\calR_A$ is not functorial in the ring $A$, it is functorial in the dagger lift $S$.
\end{defn}

\begin{lemma} \label{L:Robba properties}
With notation as in Definition~\ref{D:Robba functoriality}, 
if $\overline{g}$ is surjective (resp.\ finite), then so is the induced map $\calR_{A_1} \to \calR_{A_2}$.
\end{lemma}
\begin{proof}
By Remark~\ref{R:twist structure morphism2}, we may reduce to the case where $g$ is norm-scaling with parameter $t$.
Suppose that $\overline{g}$ is surjective.
By Lemma~\ref{L:lift properties}, for $r>0$ sufficiently small, the morphism $S_1^{tr} \to S_2^r$ is surjective, and hence strict surjective by the open mapping theorem.
By Lemma~\ref{L:dagger to integral Robba}, $g$ may be reinterpreted as the morphism $\calR_{A_1}^{\inte} \to \calR_{A_2}^{\inte}$, so for $r>0$ sufficiently small,
the morphism $\calR^{tr}_{A_1} \to \calR^r_{A_2}$ is surjective. It follows that $\calR_{A_1} \to \calR_{A_2}$ is surjective. If $\overline{g}$ is finite, a similar argument
shows that $\calR_{A_1} \to \calR_{A_2}$ is finite.
\end{proof}
\begin{cor} \label{C:Robba base change}
With notation as in Definition~\ref{D:Robba functoriality}, 
if $\overline{g}$ is finite, then
the morphism $\calR_{A_1} \otimes_{\calR_{A_1}^{\inte}} \calR_{A_2}^{\inte} \to \calR_{A_2}$ is an isomorphism.
\end{cor}

\begin{lemma} \label{L:Robba flat}
With notation as in Definition~\ref{D:Robba ring}, the morphism $\calR^{\inte}_A \to \calR_A$ is flat, and the morphism $\calR^{\bd}_A \to \calR_A$ is faithfully flat.
\end{lemma}
\begin{proof}
Using Corollary~\ref{C:Robba base change}, we reduce to the case where
\[
A = k((\overline{\pi})) \langle \overline{T}_1,\dots,\overline{T}_m \rangle
\llangle \overline{U}_1,\dots,\overline{U}_n \rrangle, \qquad
S = \bA \langle T_1,\dots,T_m \rangle \llangle U_1,\dots,U_n \rrangle.
\]
Let $R$ be the $(\frakm, \pi)$-adic completion of $\frako[\pi,T_1,\dots,T_m]\llbracket U_1,\dots,U_n \rrbracket$. 
For $r = s/t \in \QQ_{>0}$, the map $S^r \to \calR^r_A$ is flat because it can be obtained by 
taking the flat morphism $R[\varpi^s \pi^{-t}] \to R[\varpi^s \pi^{-t}, \varpi^{-s} \pi^t]$ of noetherian rings, taking $\varpi$-adic completions (which preserves flatness by
Lemma~\ref{L:flat completion}), and then inverting $\pi$.
By taking direct limits, we deduce that $\calR^{\inte}_A \to \calR_A$ is flat,
which implies immediately that $\calR^{\bd}_A \to \calR_A$ is flat.

Consequently, to check that $\calR^{\bd}_A \to \calR_A$ is faithfully flat, in light of \cite[Tag~00HQ]{stacks-project}
it remains to check  that every maximal ideal $I \in \Spec \calR^{\bd}_A$ generates a nontrivial ideal in $\calR_A$.
To see this, note that $J = I \cap \calR^{\inte}_A$ is a maximal ideal
not contained in $\frakm \calR^{\inte}_A$. Let $\overline{J}$ be the image of $J$ in $A$; by Lemma~\ref{L:mixed Nullstellensatz}, $A/J$ is a finite extension of $k((\overline{\pi}))$. But $\calR^{\bd}_{A/J}$ is a field, so $\calR^{\bd}_{A/J} \to \calR_{A/J}$ is evidently faithfully flat. This completes the proof.
\end{proof}

\begin{cor} \label{C:same units}
The map $(\calR_A^{\bd})^{\times}\to \calR_A^\times$ is an isomorphism.
\end{cor}
\begin{proof}
By Lemma~\ref{L:Robba flat} the map in question is injective,
so we need only check surjectivity.
Suppose that $x,y \in \calR_A^\times$ are inverses of each other. By 
Remark~\ref{R:Hadamard}, the functions $s \mapsto \log \left| x \right|_s$, $s \mapsto \log \left| y \right|_s$
are both convex, but their sum is zero; hence they must both be constant. It follows that
$x,y \in \calR_A^{\bd}$.
\end{proof}

\begin{cor} \label{C:bounded subring is characteristic}
For $i=1,2$, let $A$ be a semiaffinoid algebra over $k((\overline{\pi}))$ 
and let $f: \calR_{A_1} \to \calR_{A_2}$ be a ring homomorphism (not necessarily constructed as in Definition~\ref{D:Robba functoriality}).
\begin{enumerate}
\item[(a)]
The map $f$ carries $\calR_{A_1}^{\bd}$ into $\calR_{A_2}^{\bd}$.
\item[(b)]
If $f$ is continuous for the LF topologies, then $f$ carries $\calR_{A_1}^{\inte}$ into $\calR_{A_2}^{\inte}$.
In particular, this induced map is a dagger morphism which gives rise to $f$ as in Definition~\ref{D:Robba functoriality}.
\end{enumerate}
\end{cor}
\begin{proof}
By Corollary~\ref{C:same units}, $f$ carries $(\calR_{A_1}^{\bd})^\times$ into  $(\calR_{A_2}^{\bd})^\times$.
For any $x \in \calR_A^{\bd}$, for $n$ a sufficiently large integer, both $x+\varpi^{-n}$ is a unit in
$\calR_{A_1}^{\bd}$; consequently, $f(x+\varpi^{-n})$ and $f(\varpi^{-n})$  are units in
$\calR_{A_2}^{\bd}$, and their difference $f(x)$ thus belongs to $\calR_{A_2}^{\bd}$. This proves (a).
 
Suppose now that $f$ is continuous for the LF topologies.
Let $x \in \calR_{A_1}^{\inte}$ be an element whose image in $A_1$ is a topologically nilpotent unit.
Then for each positive integer $n$, $x^{-n} \varpi$ is topologically nilpotent and so $f(x)^{-n} f(\varpi)$ is also.
Since $f(x)$ is topologically nilpotent, it belongs to $\calR_{A_2}^{\inte}$; by the previous sentence, it cannot belong to $\varpi \calR_{A_2}^{\inte}$. By Lemma~\ref{L:units in integral}, $f(x)$ must be a unit in $\calR_{A_2}^{\inte}$.

Now let $y \in \calR_{A_1}^{\inte}$ be arbitrary. For $n$ a sufficiently large integer, $x^{-n} y$ is again topologically nilpotent, so $f(x)^{-n} f(y) \in \calR_{A_2}^{\inte}$. Since $f(x)$ is a unit in this ring,
$f(y) \in \calR_{A_2}^{\inte}$ also; this proves (b).
\end{proof}

\begin{defn} \label{D:Robba truncate}
With notation as in Definition~\ref{D:Robba ring}, 
for $0 < s \leq r$, let $\calR^{[s,r]}_A$ be the Fr\'echet completion of $\calR^{\inte}_A[\varpi^{-1}]$ for the seminorms $\left| \bullet \right|_t$ for $t \in [s,r]$. By Remark~\ref{R:Hadamard}, $\calR^{[s,r]}_A$ is actually a Banach ring for the norm $\max\{\left| \bullet \right|_s, \left| \bullet \right|_r\}$.
In case $r,s \in \QQ$, the ring $\calR^{[s,r]}_A$ is itself a semiaffinoid algebra over $K$,
and in particular is noetherian.

For $0 < s \leq s' \leq r' \leq r$ with $r,s,r',s' \in \QQ$, the map $\calR_A^{[s,r]} \to \calR_A^{[s',r']}$ is a semiaffinoid subdomain localization, and hence is flat 
(see Definition~\ref{D:semiaffinoid spaces}).
This means that $\calR_A^r$ is a \emph{Fr\'echet-Stein algebra} in the sense of 
Schneider--Teitelbaum \cite{schneider-teitelbaum}; in particular,
for $0 < s \leq r$ with $r,s \in \QQ$, 
the homomorphism $\calR_A^r \to \calR_A^{[s,r]}$ is flat by \cite[Remark~4.2]{schneider-teitelbaum}.

Note that the definition of the rings $\calR^{[s,r]}_A$ depends on the choice of a presentation of $S$, but there is a sense in which may compare the definition across presentations: in light of Corollary~\ref{C:same norm}, given two such presentations we can find some $r_0 > 0$ such that for $0 < s\leq r \leq r_0$, the rings $\calR^{[s,r]}_A$ arising from the two presentations coincide.
\end{defn}

\begin{lemma} \label{L:intersection}
With notation as in Definition~\ref{D:Robba truncate}, 
for $0 < t \leq s \leq r$, inside $\calR^{[s,s]}_A$ we have
\[
\calR^{[t,s]}_A \cap \calR^{[s,r]}_A = \calR^{[t,r]}_A.
\]
\end{lemma}
\begin{proof}
We first apply Remark~\ref{R:Hadamard} to see that 
the maps $\calR^{[t,s]}_A \to \calR^{[s,s]}_A$,
$\calR^{[s,r]}_A \to \calR^{[s,s]}_A$ are injective.
Explicitly, if $x \in \calR^{[t,s]}_A$ satisfies $\left| x \right|_s = 0$,
then Remark~\ref{R:Hadamard} implies that $\left| x \right|_u = 0$ for all $u \in (t,s]$,
hence also for $u=t$ by continuity, thus yielding $x = 0$.

To check the claimed inequality, we may argue as in \cite[Lemma~5.2.10]{kedlaya-liu1}
once we show that elements of $\calR^{[s,s]}_A$ can be decomposed in the manner of \cite[Lemma~5.2.8]{kedlaya-liu1}. To achieve the latter, we may reduce to the case $A = k((\overline{\pi})) \langle \overline{T}_1,\dots,\overline{T}_m \rangle \llangle \overline{U}_1, \dots, \overline{U}_n \rrangle$; in that setting, elements of $\calR^{[s,s]}_A$ may be written as formal sums in powers of $\pi$, and the desired decomposition is obtained by separating these sums depending on the sign of the exponent.
\end{proof}

\begin{prop} \label{P:coherent glueing}
Let $\Mod^{\fin}_R$ denote the category of finitely generated modules over the ring $R$.
With notation as in Definition~\ref{D:Robba truncate}, 
for $0 < t \leq s \leq r$, the functor
\[
\Mod^{\fin}_{\calR^{[t,r]}_A} \to \Mod^{\fin}_{\calR^{[t,s]}_A}
\times_{\Mod^{\fin}_{\calR^{[s,s]}_A}} \Mod^{\fin}_{\calR^{[s,r]}_A}
\]
is an equivalence of categories.
\end{prop}
\begin{proof}
For $A$ an affinoid algebra, this follows from Kiehl's theorem on coherent sheaves on affinoid spaces
\cite{kiehl}.
In the general case, this argument is not available (see Remark~\ref{R:Kiehl}).
However, by Lemma~\ref{L:intersection}, the diagram
\[
\xymatrix{
\calR^{[t,r]}_A \ar[r] \ar[d] & \calR^{[t,s]}_A \ar[d] \\
\calR^{[s,r]}_A \ar[r] & \calR^{[s,s]}_A
}
\]
forms a \emph{glueing square} in the sense of \cite[Definition~2.7.3]{kedlaya-liu1}.
Consequently, \cite[Lemma~2.7.4]{kedlaya-liu1} implies that for any object
\[
M^{[t,s]} \rightarrow M^{[s,s]} \leftarrow M^{[s,r]} \in \Mod^{\fin}_{\calR^{[t,s]}_A}
\times_{\Mod^{\fin}_{\calR^{[s,s]}_A}} \Mod^{\fin}_{\calR^{[s,r]}_A},
\]
the module
\[
M = \ker(M^{[t,s]} \oplus M^{[s,r]} \to M^{[s,s]})
\]
is finitely generated and the induced maps 
\[
M \otimes_{\calR^{[t,r]}_A} \calR_A^{[t,s]} \to M^{[t,s]}, \qquad
M \otimes_{\calR^{[t,r]}_A} \calR_A^{[s,r]} \to M^{[s,r]}
\]
are surjective. Since semiaffinoid localizations are known to be flat, 
taking the kernels of these maps gives rise to another object in the fiber product category, so we may repeat the argument once to deduce the claim (as in the proof of \cite[Lemma 2.3.2]{kedlaya-liu2}).
\end{proof}

\begin{remark}
If $A$ is an affinoid algebra, then one can also use Kiehl's theorem to glue finitely generated $\calR^{[s,r]}_A$-modules over a covering of $\sSp A$.
In the semiaffinoid case, this is not possible on account of
Remark~\ref{R:Kiehl}; it may still be possible to glue finite projective modules,
but the arguments of \cite{kedlaya-liu1} on this topic do not suffice for this.
\end{remark}

\section{Frobenius lifts and $\varphi$-modules}
\label{sec:Frobenius}

\begin{hypothesis} \label{H:Frobenius}
For the remainder of the paper, assume that $k$ is perfect of characteristic $p>0$.
Let $A$ be a nonzero semiaffinoid algebra over $k((\overline{\pi}))$, and fix a dagger lift $S$ of $A$. Let $\calE^{\inte}_A$ denote the $\varpi$-adic completion of $S$,
and put $\calE_A = \calE^{\inte}_A[\varpi^{-1}]$.
\end{hypothesis}

\begin{defn} \label{D:relative Frobenius}
Let $q$ be a power of $p$.
By a \emph{relative $q$-power Frobenius map} on $A$,
we will mean a continuous endomorphism $\overline{\varphi}$ of $A$ which is finite flat
and has the property that $\overline{\pi}^{-q} \overline{\varphi}(\overline{\pi})$ and its inverse are both power-bounded.
\end{defn}

\begin{remark} \label{R:Kunz}
Suppose that $\overline{\varphi}$ is the absolute $q$-power Frobenius endomorphism of $A$.
Then $\overline{\varphi}$ is a relative $q$-power Frobenius map in the sense of
Definition~\ref{D:relative Frobenius} if and only if it is finite flat.
By Kunz's criterion \cite{kunz}, this is true if and only if $A$ is regular;
in particular, this holds if $A$ is a smooth affinoid algebra over $k((\overline{\pi}))$ \cite[\S 7.3.2]{bgr}.
\end{remark}

\begin{remark}
Any relative $q$-power Frobenius map on $A$ extends uniquely to a relative $q$-power Frobenius map on $A \langle \overline{T}_1,\dots,\overline{T}_m \rangle \llangle \overline{U}_1,\dots,\overline{U}_n \rrangle$ fixing $\overline{T}_1,\dots,\overline{T}_m,\overline{U}_1,\dots,\overline{U}_n$.
\end{remark}

\begin{defn} \label{D:Frobenius lift}
By a \emph{relative $q$-power Frobenius lift} on $\calR_A$, we will
mean an endomorphism $\varphi$ of $\calR_A$ which is continuous for the LF topology
and acts on $\calR^{\inte}_A$ in such a way that the induced action on $A$
is a relative $q$-power Frobenius map in the sense of Definition~\ref{D:relative Frobenius}. In case the latter is the $q$-power (absolute) Frobenius map (which by Remark~\ref{R:Kunz} requires $A$ to be regular), we say that $\varphi$ is an 
\emph{absolute $q$-power Frobenius lift}.
Note that the action on $\calR_A^{\inte}$ of a relative $q$-power Frobenius lift 
on $\calR_A$ is norm-scaling with parameter $q$.
\end{defn}

\begin{remark} \label{R:Frobenius flat}
By Lemma~\ref{L:lift properties}, any relative $q$-power Frobenius lift $\varphi$ on $\calR_A$ is itself a finite flat ring homomorphism. Moreover, given a presentation of $S$, we can choose some $r_0 > 0$ such that for $0 < s \leq r \leq r_0$, $\varphi$ induces morphisms
$\calR^r_A \to \calR^{r/q}_A$, $\calR^{[s,r]}_A \to \calR^{[s/q,r/q]}_A$.
These morphisms are in fact isometric for $\left| \bullet \right|_r$ on the source and $\left|\bullet \right|_{r/q}$ on the target.
\end{remark}

\section{Perfect Robba rings}
\label{sec:perfect}

Just as the classical theory of $(\varphi, \Gamma)$-modules has been greatly elucidated by modern insights emerging from the theory of perfectoid spaces (see \cite{kedlaya-newphigamma}), it is helpful for certain purposes to make the ``perfect closure'' of a Robba ring equipped with a Frobenius lift.

\begin{hypothesis} \label{H:absolute Frobenius lift}
Throughout \S\ref{sec:perfect},
suppose that $A$ is regular.
Fix a power $q$ of $p$ and a relative $q$-power Frobenius lift $\varphi$ on $\calR_A$,
and fix $r_0>0$ as in Remark~\ref{R:Frobenius flat}.
\end{hypothesis}

\begin{defn} \label{D:perfection}
Let $\varphi$ be a relative $q$-power Frobenius lift on $A$ for some power $q$ of $p$.
Let $\breve{\calE}^{\inte}_A$ be the direct limit of the system
\[
\calE^{\inte}_A \stackrel{\varphi}{\to} \calE^{\inte}_A \stackrel{\varphi}{\to} \cdots.
\]
Let $\tilde{\calE}^{\inte}_A$ be the completion of $\breve{\calE}^{\inte}_A$ for the weak topology.
Put $\breve{\calE}_A := \breve{\calE}^{\inte}_A[\varpi^{-1}]$, $\tilde{\calE}_A := \tilde{\calE}_A^{\inte}[\varpi^{-1}]$.

For $r \in (0,r_0]$, by Corollary~\ref{C:dagger compatibility} we may form the direct system
\[
\calR_A^{\inte,r} \stackrel{\varphi}{\to} \calR_A^{\inte,r/q} \stackrel{\varphi}{\to} \calR_A^{\inte,r/q^2} \stackrel{\varphi}{\to} \cdots
\]
and these maps are isometries for the norm $\left| \bullet \right|_{r/q^n}$ on $\calR_A^{\inte,r/q^n}$. 
Let $\breve{\calR}_A^{\inte,r}$ be the direct limit,
let $\tilde{\calR}_A^{\inte,r}$ be the completion of $\breve{\calR}_A^{\inte,r}$, and 
put 
\begin{gather*}
\breve{\calR}_A^{\inte} := \bigcup_{r>0} \breve{\calR}_A^{\inte,r}, \qquad
\tilde{\calR}_A^{\inte} := \bigcup_{r>0} \tilde{\calR}_A^{\inte,r}, \\
\breve{\calR}_A^{\bd} := \breve{\calR}_A^{\inte}[\varpi^{-1}], \qquad
\tilde{\calR}_A^{\bd} := \tilde{\calR}_A^{\inte}[\varpi^{-1}].
\end{gather*}

Similarly, for $r \in (0,r_0]$, we may form the direct system
\[
\calR_A^{r} \stackrel{\varphi}{\to} \calR_A^{r/q} \stackrel{\varphi}{\to} \calR_A^{r/q^2} \stackrel{\varphi}{\to} \cdots
\]
and for each $s \in (0,r]$, these maps are isometries for the norm $\left| \bullet \right|_{s/q^n}$ on $\calR_A^{\inte,r/q^n}$. 
Let $\breve{\calR}_A^{r}$ be the direct limit,
Let $\tilde{\calR}_A^{r}$ be the Fr\'echet completion of $\breve{\calR}_A^{r}$ for the induced norms
and
put 
\[
\breve{\calR}_A := \bigcup_{r>0} \breve{\calR}_A^{r}, \qquad
\tilde{\calR}_A := \bigcup_{r>0} \tilde{\calR}_A^{r}
\]
\end{defn}

\begin{remark} \label{R:Jacobson radical tilde}
By Remark~\ref{R:Jacobson radical}, $\varpi$ is contained in the Jacobson radical of $\calR_A^{\inte}$.
By similar logic, $\varpi$ is also contained in the Jacobson radical of $\tilde{\calR}_A^{\inte}$.
\end{remark}

\begin{lemma} \label{L:Robba split}
The following statements hold.
\begin{enumerate}
\item[(a)]
For $r\in(0,r_0]$, the morphisms 
\[
\calR_A^{\inte,r} \to \tilde{\calR}_A^{\inte,r}, \qquad \calR^r_A \to \tilde{\calR}^r_A
\]
are faithfully flat and module-split.
\item[(b)]
The morphisms 
\[
\calR_A^{\inte} \to \tilde{\calR}_A^{\inte}, \quad 
\calR_A^{\bd} \to \tilde{\calR}_A^{\bd}, \quad 
\calR_A \to \tilde{\calR}_A, \quad \calE^{\inte}_A \to \tilde{\calE}^{\inte}_A, \quad \calE_A \to \tilde{\calE}_A
\]
are faithfully flat and module-split.
\item[(c)]
The morphisms
\[
\calR_A^{\inte} \to \calE_A^{\inte}, \breve{\calR}_A^{\inte} \to \breve{\calR}_A^{\inte},
\tilde{\calR}_A^{\inte} \to \tilde{\calE}_A^{\inte},
\calR_A^{\bd} \to \calE_A^{\bd},
\breve{\calR}_A^{\bd} \to \breve{\calE}_A^{\bd},
\tilde{\calR}_A^{\bd} \to \tilde{\calE}_A^{\bd}
\]
are faithfully flat.
\end{enumerate}
\end{lemma}
\begin{proof}
By Lemma~\ref{L:lift properties} and Remark~\ref{R:finite generators}, 
$\varphi: \calR_A^{\inte,r} \to \calR_A^{\inte,r/q}$ is module-split and its cokernel
can be written as a direct summand of a finite free $\calR_A^{\inte,r}$-module via $\varphi$.
By iterating this construction, we see that $\tilde{\calR}_A^{\inte,r}$ can be written as a direct 
summand of a module over $\calR_A^{\inte,r}$ which is a completion of a free module, and in particular is faithfully flat; we also obtain a module splitting of $\calR_A^{\inte,r} \to \tilde{\calR}_A^{\inte,r}$. 
A similar argument applies to $\calR_A^r \to \tilde{\calR}_A^r$. This yields (a). 

For (b), the first two assertions follow at once from (a); the last two follow by a similar argument as in (a).

For (c), note that in the cases $\calR_A^{\inte} \to \calE_A^{\inte}, \tilde{\calR}_A^{\inte} \to \tilde{\calE}_A^{\inte}$, we are forming the completion with respect to the finitely generated ideal $(\varpi)$ which is in the Jacobson radical (Remark~\ref{R:Jacobson radical tilde}), so the resulting map is faithfully flat.
This implies that $\calR_A^{\bd} \to \calE_A^{\bd}$, $\tilde{\calR}_A^{\bd} \to \tilde{\calE}_A^{\bd}$ are faithfully flat by inverting $\varpi$, and that $\breve{\calR}_A^{\inte} \to \breve{\calR}_A^{\inte}$,
$\breve{\calR}_A^{\bd} \to \breve{\calE}_A^{\bd}$ are faithfully flat by taking direct limits.
\end{proof}

\begin{lemma} \label{L:descend submodule}
Let $R_1 \to R_2$ be one of the following morphisms:
\[
\calE^{\inte}_A \to \tilde{\calE}^{\inte}_A,
\quad \calE_A \to \tilde{\calE}_A,
\quad \calR^{\inte}_A \to \tilde{\calR}^{\inte}_A,
\quad \calR_A^{\bd} \to \tilde{\calR}_A^{\bd},
\quad \calR_A \to \tilde{\calR}_A.
\]
Let $M$ be a finite projective module over $R_1$, put
$\tilde{M} = M \otimes_{R_1} R_2$, and let
$\tilde{N}$ be a direct summand of $\tilde{M}$.
Then the intersection $N = \tilde{N} \cap M$ within $\tilde{M}$
 is a direct summand of $M$. In particular, $N$ is a finite projective module over $R_1$ and the induced map
$N \otimes_{R} R_2 \to \tilde{N}$ is an isomorphism.
\end{lemma}
\begin{proof}
From Lemma~\ref{L:Robba split}, we obtain an $R_1$-linear splitting of the inclusion $R_1 \hookrightarrow R_2$. By tensoring with $M$, we obtain a splitting $\tilde{M} \to M$ of the inclusion $M \hookrightarrow \tilde{M}$; this splitting in turn restricts to a splitting $\tilde{N} \to N$ of the inclusion $N \to \tilde{N}$. 

Since $\tilde{N}$ is a direct summand of $\tilde{M}$, there exists an $R_2$-linear splitting $\tilde{M} \to \tilde{N}$ of the 
inclusion $\tilde{N} \hookrightarrow \tilde{M}$.
The composition $M \to \tilde{M} \to \tilde{N} \to N$
is then an $R_1$-linear splitting of the inclusion $N \hookrightarrow M$.
\end{proof}

\begin{lemma} \label{L:descend projective}
Any finite projective module over $\tilde{\calR}^{\inte}_A$ is the base extension of some finite projective
module over $\breve{\calR}^{\inte}_A$.
\end{lemma}
\begin{proof}
(Compare \cite[Lemma~5.6.8]{kedlaya-liu2}.)
Let $P$ be a finite projective module over $\tilde{\calR}^{\inte}_A$; we may assume $P \neq 0$. We can then find another
finite projective module $Q$ over $\tilde{\calR}^{\inte}_A$ such that $F = P \oplus Q$ is free.
The composition $F \to P \to F$ is a projector defined by some matrix $U$.
Choose $r>0$ for which $U$ is defined over $\tilde{\calR}^{\inte,r}_A$;
then $\left| U \right|_r \geq 1$.
For some positive integer $n$,
we can choose a matrix $V$ over $\varphi^{-n}(\calR^{\inte,r/q^n}_A)$ of the same size as $U$ such that $\left| V-U \right|_r < \left|U\right|_r^{-3}$.
In particular, since $V^2-V = V^2-V-U^2+U = (V-U)(V+U-1)$, we have $\left|V^2-V\right|_r < \left|U\right|_r^{-2}$.

We now define a sequence of matrices $W_0, W_1,\dots$ by taking $W_0 = V$ and forming $W_{l+1}$ from $W_l$ via a modified Newton iteration:
\[
W_{l+1} = 3W_l^2 - 2W_l^3.
\]
Since
\begin{align*}
W_{l+1} - W_l &= (W_l^2 - W_l)(1 - 2W_l) \\
W_{l+1}^2 - W_{l+1} &= (W_l^2-W_l)^2 (4W_l^2 - 4W_l - 3), 
\end{align*}
by induction on $l$ we have
\begin{align*}
\left| W_l -U \right|_r &< \left| U \right|_r^{-2} \\
\left| W_l^2 - W_l \right|_r &\leq \left| U \right|_r^{-2} \left( \left| V^2 - V \right|_r \left| U \right|_r^2\right)^{2^l}.
\end{align*}
Since $\varphi^{-n}(\calR^{\inte,rq^{-n}}_A)$ is complete with respect to
$\left| \bullet \right|_r$, 
the $W_n$ converge to a matrix $W$ over $\breve{\calR}^{\inte,r}_A$
with $\left |W-U \right|_r < \left |U\right|_r^{-2}$ and $W^2 = W$.
Let $F_0$ be a free module over $\breve{\calR}^{\inte,r}_A$ equipped with an isomorphism
$F_0 \otimes_{\breve{\calR}^{\inte,r}_A} \tilde{\calR}^{\inte}_A \cong F$.
Let $P_0$ and $Q_0$ be the images of the projector on $F_0$ defined by $W$ and $1-W$.
We then obtain a map $P_0 \otimes_{\breve{\calR}^{\inte,r}_A} \tilde{\calR}^{\inte}_A
\to P$ by embedding $P_0$ into $F_0$ and then projecting $F$ onto $P$;
we similarly obtain a map $Q_0 \otimes_{\breve{\calR}^{\inte,r}_A} \tilde{\calR}^{\inte}_A
\to Q$. Taking the direct sum, we obtain the map $F \to F$ defined by the matrix
\begin{align*}
UW + (1-U)(1-W) &= UW - U^2 - W^2 + UW  + 1\\
&= U(W-U) - (W-U)W + 1,
\end{align*}
for which we see that $|UW+(1-U)(1-W)-1|_r < 1$.
The map is therefore an isomorphism, so in particular there exists an isomorphism
$P_0 \otimes_{\breve{\calR}^{\inte,r}_A} \tilde{\calR}^{\inte}_A
\cong P$.
\end{proof}

\begin{remark}
Suppose that $\varphi$ is an absolute $q$-power Frobenius lift.
In this case, the ring $\tilde{\calR}_A$ can be identified with the relative Robba ring $\tilde{\calR}_R$ considered in  \cite{kedlaya-liu1, kedlaya-liu2} by taking $R$ to be the completed perfect closure of $A$.
In this way, we may apply results from \textit{op. cit.} to $\varphi$-modules over $\tilde{\calR}_A$.
\end{remark}

\section{$\varphi$-modules}
\label{sec:phi-modules}

Throughout \S\ref{sec:phi-modules}, continue to retain Hypothesis~\ref{H:absolute Frobenius lift}.

\begin{remark} \label{R:projective module topology}
In what follows, we will freely use the fact that any finite projective module over a normed ring carries a unique norm topology, obtained by 
choosing a presentation of the module as a direct summand of a finite free module and imposing the supremum norm induced by some basis of the latter. For a detailed derivation, see \cite[Lemma~2.2.12]{kedlaya-liu1}.
\end{remark}

\begin{defn}
Let $R$ be one of the rings in the following commutative diagram (in which all morphisms are inclusions):
\begin{equation} \label{eq:ring diagram}
\xymatrix@R10pt@C10pt{
\calR^{\inte}_A \ar[ddd] \ar[rrr] \ar[rd] &&& \calR^{\bd}_A \ar[rrr] \ar[ddd]\ar[rd] &&& \calR_A \ar[rd]\\
& \breve{\calR}^{\inte}_A \ar[ddd] \ar[rrr] \ar[rd] &&& \breve{\calR}^{\bd}_A \ar[rrr] \ar[ddd] \ar[rd] &&& \breve{\calR}_A \ar[rd] \\
&& \tilde{\calR}^{\inte}_A \ar[ddd] \ar[rrr] &&& \tilde{\calR}^{\bd}_A \ar[rrr] \ar[ddd] &&& \tilde{\calR}_A \\
\calE^{\inte}_A \ar[rrr] \ar[rd] &&& \calE_A \ar[rd] \\
& \breve{\calE}^{\inte}_A \ar[rrr] \ar[rd] &&& \breve{\calE}_A \ar[rd] \\
&& \tilde{\calE}^{\inte}_A \ar[rrr] &&& \tilde{\calE}_A.
}
\end{equation}
A \emph{projective $\varphi$-module} 
(resp.\ \emph{coherent $\varphi$-module})
over $R$ is a finite projective (resp.\ finitely presented)
$R$-module $M$ equipped with a semilinear action of $\varphi$ for which the induced map $\varphi^* M \to M$ is an isomorphism. Note that what we call a projective $\varphi$-module is often called simply a \emph{$\varphi$-module}.
\end{defn}

It is useful to reformulate the definition of a $\varphi$-module in more geometric terms.
\begin{lemma} \label{L:bundle to module}
Choose $0 < s \leq r \leq r_0$ with $r,s \in \QQ$. Let $N$ be a finitely generated (resp.\ finite projective) module over $\calR_A^{[s/q,r]}$ equipped with an isomorphism 
\[
\varphi^* N \otimes_{\calR_A^{[s/q^2,r/q]}} \calR_A^{[s/q,r/q]} \cong N \otimes_{\calR_A^{[s/q,r]}} \calR_A^{[s/q,r/q]}.
\]
Then $N$ lifts uniquely to a finitely generated (resp.\ finite projective) module $M_r$ over $\calR_A^r$ equipped with an isomorphism 
\[
\varphi^* M_r \cong M_r \otimes_{\calR_A^{r}} \calR_A^{r/q}.
\]
In particular, $N$ functorially gives rise to the coherent (resp.\ projective) $\varphi$-module  $M_r \otimes_{\calR_A^r} \calR_A$ over $\calR_A$.
\end{lemma}
\begin{proof}
For $n$ a nonnegative integer, let $N^{[s/q^{n+1}, r/q^n]}$ be the pullback of $N$ along
$\varphi^n$, viewed as a module over $\calR_A^{[s/q^{n+1}, r/q^n]}$ which is generated by a fixed number of elements (independent of $n$). 
Using Proposition~\ref{P:coherent glueing}, we may glue these modules together to obtain
a family of coherent modules over $\calR^{[s',r']}_A$ for all $r',s' \in \QQ$
with $0 < s' \leq r' \leq r$. By the analogue of Cartan's Theorem~A and Theorem~B for Fr\'echet-Stein algebras (see \cite[\S 3]{schneider-teitelbaum}), these modules all arise as base extensions from their inverse limit $N$, which is a \emph{coadmissible} module over $\calR_A^r$. However, using the uniform bound on the number of generators and relations needed to generate 
the $N^{[s/q^{n+1}, r/q^n]}$, we may argue as in \cite[Proposition~2.2.7]{kpx}
to see that $N$ is finitely presented (resp.\ finite projective);
when  $A$ is affinoid, this step can also be treated using an argument of Bellovin, for which see  \cite[\S2.2]{bellovin} or \cite[Proposition~2.6.16]{kedlaya-liu2}.
\end{proof}

\begin{remark}
When $A$ is affinoid, Lemma~\ref{L:bundle to module} asserts that a coherent (resp.\ projective) $\varphi$-module over $\calR_A$ is the same thing as a coherent (resp.\ coherent and locally free) sheaf on the rigid analytic space obtained by pasting together the
affinoid spaces $\sSp \calR^{[s,r]}_A$ for $0 < s \leq r \leq r_0$ for some choice of $r_0$.
\end{remark}

\section{\'Etale $\varphi$-modules}
\label{sec:etale phi-modules}

Throughout \S\ref{sec:etale phi-modules}, continue to retain Hypothesis~\ref{H:absolute Frobenius lift}.

\begin{defn} \label{D:etale phi-module}
For $R$ a ring appearing in \eqref{eq:ring diagram}, let $R_0$ be the subring appearing farthest to the left of $R$ in the same row. 
We say that a projective $\varphi$-module $M$ over $R$ 
is \emph{\'etale} if it arises by base extension from a $\varphi$-module $M_0$ over $R_0$; 
we then say that $M_0$ is an \emph{\'etale model} of $M$.
Note that this condition is vacuous when $R = R_0$, i.e., when $R = \calR_A^{\inte}, \breve{\calR}_A^{\inte}, \tilde{\calR}_A^{\inte}, \calE_A^{\inte}, \breve{\calE}_A^{\inte}, \tilde{\calE}_A^{\inte}$.
\end{defn}

\begin{lemma} \label{L:phi-module structure on projective}
For any ring $R$ in \eqref{eq:ring diagram}, every projective $R$-module can be equipped with the structure of a projective $\varphi$-module over $R$.
\end{lemma}
\begin{proof}
Let $M$ be a finite projective $R$-module. Choose a finite free $R$-module $F$ admitting a direct sum decomposition $F \cong M \oplus N$. We may equip $F$ with the structure of a projective $\varphi$-module by fixing an isomorphism $F \cong R^{\oplus n}$ of $R$-modules and equipping $R^{\oplus n}$ with the trivial $\varphi$-module structure.
We may then specify a $\varphi$-action on $N$ as the composition $N \to F \stackrel{\varphi}{\to} \to F \to N$.
\end{proof}

\begin{lemma} \label{L:phi-module lift to projective}
For any ring $R$ in \eqref{eq:ring diagram}, every coherent $\varphi$-module over $R$ can be written as a quotient of some projective $\varphi$-module over $R$; in fact we can even ensure that the underlying $R$-module of the latter is not only projective, but finite free over $R$.
\end{lemma}
\begin{proof}
This follows from a general assertion about finitely generated modules over a ring equipped with a semilinear action of some endomorphism of the base ring; see \cite[Lemma~1.5.2]{kedlaya-liu1}.
\end{proof}

\begin{cor} \label{C:phi-module projective to free}
For any ring $R$ in \eqref{eq:ring diagram}, every projective $\varphi$-module over $R$ is a direct summand of a 
projective $\varphi$-module over $R$ whose underlying $R$-module is free.
\end{cor}
\begin{proof}
This follows from Lemma~\ref{L:phi-module structure on projective} and Lemma~\ref{L:phi-module lift to projective}
as in \cite[Corollary~1.5.3]{kedlaya-liu1}.
\end{proof}

\begin{lemma} \label{L:tilde preserves phi-cohomology}
Let $M$ be a finitely generated module over $\calE^{\inte}_A$ equipped with a $\calE^{\inte}_A$-linear
morphism $\varphi^* M \to M$. (Unlike in the definition of a $\varphi$-module, we do not require this morphism to be an isomorphism.)
Put
\[
\breve{M} := M \otimes_{\calE^{\inte}_A} \breve{\calE}^{\inte}_A,
\qquad
\tilde{M} := M \otimes_{\calE^{\inte}_A} \tilde{\calE}^{\inte}_A.
\]
 Then the maps $M \to \breve{M} \to \tilde{M}$ induce quasi-isomorphisms between the mapping cones of $\varphi-1$ on $M, \breve{M}, \tilde{M}$ (that is, they preserve the kernel and the cokernel of $\varphi-1$).
\end{lemma}
\begin{proof}
By Lemma~\ref{L:lift properties} and Lemma~\ref{L:Robba split}, the morphisms $\calR_A^{\inte} \to \breve{\calR}_A^{\inte} \to \tilde{\calR}_A^{\inte}$ are faithfully flat. In particular, the morphisms $M \to \breve{M} \to \tilde{M}$ are injective, so the induced maps on kernels are injective.

We next verify the claim for $M \to \breve{M}$.
\begin{itemize}
\item
 If $x \in \breve{M}$ belongs to the kernel of $\varphi-1$, then $\varphi^n(x) \in M$ for some nonnegative integer $n$ and so $x = \varphi^n(x) \in M$; hence the kernel of $\varphi-1$ on $M$ surjects onto the kernel of $\varphi-1$ on $M$.
 \item
  If $x \in M$ has the form $(\varphi-1)(y)$ for some $y \in \breve{M}$, then $\varphi^n(y) \in M$ for some nonnegative integer $n$, and $\varphi^n(x) = (\varphi-1)(\varphi^n(y))$ represents the same class in the cokernel as does $x$; hence the cokernel of $\varphi-1$ on $M$ injects into the cokernel of $\varphi-1$ on $\breve{M}$. 
\item
For any $x \in \breve{M}$, we have $\varphi^n(x) \in M$ for some nonnegative integer $n$, and $x$ and $\varphi^n(x)$ represent the same class in the cokernel of $\varphi-1$ on $\breve{M}$. Hence the cokernel of $\varphi-1$ on $M$ surjects onto the cokernel of $\varphi-1$ on $M$.
\end{itemize}

We next verify the desired result in the case where $M$ is projective over $A$. In this case, there exists a neighborhood $N$ of 0 in $M$ such that $1-\varphi$ acts invertibly with the inverse $1 + \varphi + \varphi^2 + \cdots$. Since $\breve{M}$ is dense in $M$, it follows at once that $\breve{M} \to M$ induces a surjection of cokernels.
Moreover, if $\bv \in \breve{M}, \bw \in \tilde{M}$ satisfy $(1-\varphi)(\bw) = \bv$, we can write
$\bw = \bw_0 + \bw_1$ with $\bw_0 \in \breve{M}, \bw_1 \in N$, and then
\[
(1-\varphi)(\bw_1) = \bv + (\varphi-1)(\bw_0) \in \breve{M}.
\]
In particular, there is a nonnegative integer $n$ for which $(1-\varphi)(\bw_1) \in \varphi^{-n}(M)$;
we then have
\[
\bw_1 = (1 + \varphi + \varphi^2 + \cdots)((1-\varphi)(\bw_1)) \in \varphi^{-n}(M)
\]
and so $\bw \in \breve{M}$. This implies that $\breve{M} \to M$ induces a surjection of kernels (by taking $\bv = 0$) and an injection of cokernels, completing the proof in this case.

We next note that using Lemma~\ref{L:phi-module lift to projective} and the snake lemma, we may apply the previous paragraph to deduce the desired result in the case where $M$ is killed by $\varpi$. This in turn implies the case where $M$ is killed by some power of $\varpi$.

To deduce the general case, we work with the map $M \to \tilde{M}$. By the previous paragraph, we have a surjection of cokernels modulo $\varpi$; this immediately implies the surjection of cokernels. 
If $\bv \in M, \bw \in \tilde{M}$ satisfy $(1-\varphi)(\bw) = \bv$, 
then the previous paragraph again (specifically, the fact that the reduction of $M \to \tilde{M}$ modulo $\varpi^n$ induces a surjection of kernels and an injection of cokernels) implies that $\bw \in \varpi^n M + \tilde{M}$ for each
positive integer $n$. It follows that $\bv \in M$, and so we obtain the surjection of cokernels and the injection of kernels.
\end{proof}

\begin{lemma} \label{L:integral tilde base extension}
The base extension functor from coherent $\varphi$-modules over $\calE^{\inte}_A$ to coherent $\varphi$-modules over $\tilde{\calE}^{\inte}_A$ is an equivalence of categories.
\end{lemma}
\begin{proof}
The ring $\calE^{\inte}_A$ being noetherian, the category of coherent $\varphi$-modules over $\calE^{\inte}_A$ admits internal Homs. Since $\calE^{\inte}_A \to \calE_A$ is faithfully flat by Lemma~\ref{L:Robba split}(b),
the formation of internal Homs commutes with base extension along this morphism.
(In particular, internal Homs exist on the essential image of the base extension functor; we do not yet know that they exist on the entire category of $\varphi$-modules over $\tilde{\calE}^{\inte}_A$ admits internal Homs, as this will only follow \emph{a posteriori} from (a).)
Consequently, the full faithfulness assertion in (a) follows from the preservation of kernels in
Lemma~\ref{L:tilde preserves phi-cohomology}. Similarly, the preservation of cokernels in 
Lemma~\ref{L:tilde preserves phi-cohomology} implies that the formation of Yoneda extension groups commutes with the base change.

Based on the previous paragraph, we may reduce the verification of essential surjectivity first to the case where
$M$ is killed by some power of $\varpi$ (as we may then take inverse limits over reductions modulo power of $\varpi$ to deal with the general case), then to the case where $M$ is killed by $\varpi$ (using the preservation of extensions), and then finally to the case where $M$ is a free $A$-module (using Lemma~\ref{L:phi-module lift to projective}. In this case, fix a basis $\be_1,\dots,\be_d$ of $M$ and let $F$ be the matrix over $\tilde{A}$ expressing the action of $\varphi$ on this basis; then using a neighborhood $N$ as in the proof of Lemma~\ref{L:tilde preserves phi-cohomology},
we see that any sufficiently good approximation of $F$ over $\breve{A}$ defines an isomorphic $\varphi$-module.
This provides a descent of $M$ to $\breve{A}$, or more precisely to $\varphi^{-n}(A)$ for some nonnegative integer $n$; applying $\varphi^n$ yields a basis on which the action of Frobenius is defined over $A$, as desired.
\end{proof}

\begin{lemma}\label{L:etale fully faithful}
The following statements hold.
\begin{enumerate}
\item[(a)]
The base extension functor from \'etale $\varphi$-modules over $\calR^{\bd}_A$ to \'etale $\varphi$-modules over $\calE_A$ is fully faithful.
\item[(b)]
The base extension functor from \'etale $\varphi$-modules over $\calR^{\bd}_A$ to \'etale $\varphi$-modules over $\calR_A$ is an equivalence of categories.
\item[(c)]
The base extension functor from \'etale $\varphi$-modules over $\tilde{\calR}^{\bd}_A$ to \'etale $\varphi$-modules over $\tilde{\calR}_A$ is an equivalence of categories.
\end{enumerate}
\end{lemma}
\begin{proof}
As in \cite[Remark~4.3.4]{kedlaya-liu1}, (a) reduces to checking that
for $M$ a $\varphi$-module over $\calR^{\inte}_A$, we have 
\[
M^{\varphi} = (M \otimes_{\calR^{\inte}_A} \calE^{\inte}_A)^{\varphi};
\]
this may be established by following the method of \cite[Lemma~5.4.1]{kedlaya-revisited}.
Similarly, (b) reduces to checking that 
for $M$ a $\varphi$-module over $\calR^{\inte}_A$, we have 
\[
(M \otimes_{\calR^{\inte}_A} \calR^{\bd}_A)^{\varphi} = (M \otimes_{\calR^{\inte}_A} \calR_A)^{\varphi};
\]
this may be established by following the proof of \cite[Proposition~1.2.6]{kedlaya-revisited}. 
(For a similar argument, see the proof of Lemma~\ref{L:integrality} below.)
The proof of (c) is similar to that of (b).
\end{proof}

\begin{theorem}
Consider the following diagram:
\begin{equation} \label{eq:ring diagram2}
\xymatrix@R10pt@C10pt{
\calR^{\inte}_A \ar@{-->}[ddd] \ar@{==>}[rrr] \ar[rd] &&& \calR^{\bd}_A \ar[rr] \ar@{-->}[ddd]\ar[rd] && \calR_A \ar[rd]\\
& \breve{\calR}^{\inte}_A \ar@{-->}[ddd] \ar@{==>}[rrr] \ar@{-->}[rd] &&& \breve{\calR}^{\bd}_A \ar[rr] \ar@{-->}[ddd] \ar@{-->}[rd] && \breve{\calR}_A \ar@{-->}[rd] \\
&& \tilde{\calR}^{\inte}_A \ar[ddd] \ar@{==>}[rrr] &&& \tilde{\calR}^{\bd}_A \ar[rr] \ar[ddd] && \tilde{\calR}_A \\
\calE^{\inte}_A \ar@{==>}[rrr] \ar[rd] &&& \calE_A \ar[rd] \\
& \breve{\calE}^{\inte}_A \ar@{==>}[rrr] \ar[rd] &&& \breve{\calE}_A \ar[rd] \\
&& \tilde{\calE}^{\inte}_A \ar@{==>}[rrr] &&& \tilde{\calE}_A.
}
\end{equation}
\begin{enumerate}
\item[(a)]
Base extension of \'etale $\varphi$-modules along every single dashed or solid arrow in the diagram is fully faithful, and preserves the kernel and cokernel of $\varphi-1$.
\item[(b)]
Base extension of \'etale $\varphi$-modules along every double dashed arrow in the diagram induces an equivalence of isogeny categories (inverting $\varpi$).
\item[(c)]
Base extension of \'etale $\varphi$-modules along every solid arrow in the diagram is an equivalence of categories.
\end{enumerate}
\end{theorem}
\begin{proof}
Statement (a) follows from Lemma~\ref{L:etale fully faithful}(a). 
Statement (b) follows formally from the definition of an \'etale $\varphi$-module.
As for (c), we note that the arrows $* \to \breve{*}$ are trivial to handle;
the arrows $\calE_A^{\inte} \to \breve{\calE}_A^{\inte} \to \tilde{\calE}_A^{\inte}$ and
$\calE_A \to \breve{\calE}_A \to \tilde{\calE}_A$ are handled by Lemma~\ref{L:integral tilde base extension};
the arrows $\calR_A^{\bd} \to \calR_A$, $\breve{\calR}_A^{\bd} \to \breve{\calR}_A$ are handled by
Lemma~\ref{L:etale fully faithful}(b);
the arrow $\tilde{\calR}_A^{\bd} \to \tilde{\calR}_A$ is handled by 
Lemma~\ref{L:etale fully faithful}(c); and the arrows
$\tilde{\calR}_A^{\inte} \to \tilde{\calE}_A^{\inte}$, $\tilde{\calR}_A^{\bd} \to \tilde{\calE}_A$
are handled by an argument in the style of \cite[Theorem~8.5.3]{kedlaya-liu1}.
\end{proof}

\begin{remark}
For the morphisms $\calR^{\inte}_A \to \calE^{\inte}_A$, $\breve{\calR}_A^{\inte} \to \calE_A^{\inte}$,
$\calR^{\bd}_A \to \calE_A$, $\breve{\calR}^{\bd}_A \to \breve{\calE}_A$, it is expected that the base extension of \'etale $\varphi$-modules is not essentially surjective in general.
This would imply in turn that for the morphisms $\breve{\calR}_A^{\inte} \to \tilde{\calR}_A^{\inte}$,
$\breve{\calR}_A^{\bd} \to \tilde{\calR}_A^{\bd}, \breve{\calR}_A \to \tilde{\calR}_A$, base extension of \'etale $\varphi$-modules is not essentially surjective in general.

We propose a specific example which we have not verified in detail.
Take $\frako = \FF_p \llbracket \varpi \rrbracket$, $A = \FF_p((\pi))$, 
set $x = 1 + \sum_{n=1}^\infty \varpi^n \pi^{-1-p^n}$, and let $M$ be the free $\calE_A^{\inte}$-module on the single generator $\bv$ with the $\varphi$-action given by $\bv \mapsto x \bv$.
It should be possible to verify that $M$ does not descend to a $\varphi$-module over $\calR_A^{\inte}$;
this amounts to showing that for every $u \in \calE_A^{\inte \times}$, we have $u^{-1} x \varphi(u) \notin \calR_A^{\inte}$.
\end{remark}

\begin{remark} \label{R:reflect etale}
Beware that even when base extension along a morphism $R_1 \to R_2$ induces an equivalence of categories of
\'etale $\varphi$-modules, it does not immediately follow that for $M$ a projective $\varphi$-module over $R_1$, knowing that $M \otimes_{R_1} R_2$ is \'etale implies that $M$ itself is \'etale; it only follows that
there is an \'etale $\varphi$-module $N$ (which is unique up to unique isomorphism) over $R_1$ such that
$M \otimes_{R_1} R_2 \cong N \otimes_{R_1} R_2$. 

The fact that one may check the \'etale property is vacuously true in case $R_1$ is a ring of the form $*^{\inte}$, or in case $R_1 \to R_2$ is a map of the form $* \to \breve{*}$ (as in such cases base extension of arbitrary $\varphi$-modules is an equivalence of categories). The remaining cases are treated in Theorem~\ref{T:etale after base change} and Remark~\ref{R:etale after base change}.
\end{remark}

\begin{lemma} \label{L:base extension E}
The base extension functor from coherent $\varphi$-modules over $\calE_A$ to coherent $\varphi$-modules over
$\tilde{\calE}_A$ is an equivalence of categories.
\end{lemma}
\begin{proof}
Instead of the category of coherent $\varphi$-modules over either  $\calE_A^{\inte}$ or $\tilde{\calE}_A^{\inte}$,
we may work with the category of coherent modules $M$ equipped with linear morphisms $\varphi^* M \to M$ which need not be isomorphisms (we then recover the desired result by passing to the isogeny category). 
In this context, we may apply Lemma~\ref{L:tilde preserves phi-cohomology} as in the proof of 
Lemma~\ref{L:integral tilde base extension}.
\end{proof}

\begin{lemma} \label{L:restrict etale to R}
Let $M$ be a finite projective $\calR_A^{\bd}$-module and put $M' := M \otimes_{\calR_A^{\bd}} \calE_A$.
Let $M'_0$ be a finite projective $\calE_A^{\inte}$-module equipped with an isomorphism
$M'_0 \otimes_{\calE_A^{\inte}} \calE_A \cong M'$ and put $M_0 := M \cap M'_0$ within $M'$.
Then $M_0$ is a finite projective $\calR_A^{\inte}$-module.
(Moreover, the same assertion holds with each ring $*$ replaced by $\tilde{*}$.)
\end{lemma}
\begin{proof}
Since $\calR_A^{\bd}$ is dense in $\calE_A$ for the $\varpi$-adic topology, $M$ is dense in $M'$ for the $\varpi$-adic topology, and so $M_0$ is dense in $M'_0$. In particular, given a generating set $\be'_1,\dots,\be'_n$ of $M'_0$, we can find elements $\be_1,\dots,\be_n$ of $M_0$ and a matrix $A$ over $\calE_A^{\inte}$ such that $\be_j = \sum_{ij} A_{ij} \be'_i$ and $A-1$ is divisible by $\varpi$. By inverting $A$ we see that $\be_1,\dots,\be_n$ generates $M'_0$ and so $M_0 \otimes_{\calR_A^{\inte}} \calE_A^{\inte} \to M'_0$ is surjective.

Let $F_0$ be a finite free $\calR_A^{\inte}$-module on $n$ generators equipped with the $\calR_A^{\inte}$-linear morphism $F_0 \to M_0$ carrying these generators to $\be_1,\dots,\be_n$. Let $N_0$ be the kernel of $F_0 \to M_0$.
Put $F'_0 := F_0 \otimes_{\calR_A^{\inte}} \calE_A^{\inte}$ and $N' := \ker(F'_0 \to M'_0)$.
We now have a commutative diagram
\[
\xymatrix{
 & N_0 \otimes_{\calR_A^{\inte}} \calE_A^{\inte} \ar[r] \ar[d] & 
F_0 \otimes_{\calR_A^{\inte}} \calE_A^{\inte} \ar[r] \ar[d] & 
M_0 \otimes_{\calR_A^{\inte}} \calE_A^{\inte} \ar[r] \ar[d] & 0 \\
0 \ar[r] & N'_0 \ar[r] & F'_0 \ar[r] & M'_0 \ar[r] & 0
}
\]
with exact rows. Of the vertical arrows, the middle one is an isomorphism, the right one is surjective by the previous argument, and the left one is surjective by similar logic (replacing $M$ with
$N := N_0 \otimes_{\calR_A^{\inte}} \calR_A^{\bd}$); it follows that the right vertical arrow is also injective, and hence an isomorphism. Since $\calR_A^{\inte} \to \calE_A^{\inte}$ is faithully flat (Lemma~\ref{L:Robba split}(c)),
we deduce the desired result.
\end{proof}

\begin{lemma} \label{L:integrality}
Put $\calS = \tilde{\calR}_A \widehat{\otimes}_{\calR_A} \tilde{\calR}_A$
and let $\calS^{\inte}$ be the image of $\tilde{\calR}^{\inte}_A \widehat{\otimes}_{\calR_A^{\inte}} \tilde{\calR}^{\inte}_A$
in $\calS$. Let $F$ be a $d \times d$ matrix over $\calS^{\inte}$
and let $\bv$ be a column vector of length $d$ over $\calS$ such that $\bv = F \varphi(\bv)$.
Then $\bv$ has entries in $\calS^{\inte}[\varpi^{-1}]$.
\end{lemma}
\begin{proof}
We use an argument modeled on the proof of \cite[Proposition 3.3.4]{kedlaya-relative}.
As per Lemma~\ref{L:Robba split}, we can write $\calR_A$ as a completed direct sum
\[
\widehat{\bigotimes}_{n=0}^\infty M_n, \qquad M_n := \begin{cases} \calR_A & n=0 \\
\varphi^{-n}(\calR_A)/\varphi^{-n+1}(\calR_A) & n>0.
\end{cases}
\]
We may then write
\begin{equation} \label{eq:splitting S}
\calS \cong \widehat{\bigotimes}_{n_1,n_2=0}^\infty M_{n_1} \otimes_{\calR_A} M_{n_2}.
 \end{equation}
Equip $\calS$ with the restriction of the supremum norm with respect to this decomposition.
By rescaling $F$ and $\bv$ by a power of $\pi$ as needed, we may reduce to the case where
there exists $r > 0$ such that
$|xF|_s \leq |x|_s$ for all $s \in (0,r]$ and all $x \in \calS$. 
(Note that the norm on $\calS$ is not submultiplicative, so it is not enough to require that $|F|_s \leq 1$.)
Choose $c>0$ so that $|\bv|_s \leq c$ for $s \in [r/q, r]$;
then for each positive integer $m$, the equality $\bv = F \varphi(F) \cdots \varphi^{m-1}(F) \varphi^m(\bv)$
implies that 
$|\bv|_s \leq c$ for $s \in [r/q^{m+1}, r/q^m]$ by induction on $m$. It follows that $|\bv|_s \leq c$ for all $s \in (0,r]$, and so $\bv \in \calS^{\inte}$.
\end{proof}

\begin{lemma} \label{L:density}
Let $M$ be a projective $\varphi$-module over $\calR_A$ such that $\tilde{M} = M \otimes_{\calR_A} \tilde{\calR}_A$
admits a free \'etale model $\tilde{M}_0$. Put $M_0 = M \cap \tilde{M}_0$ within $\tilde{M}$.
Then $\bigcup_{n=0}^\infty \varphi^{-n}(M_0)$ is dense in $\tilde{M}_0$ for the LF topology
of $\tilde{\calR}_A$.
\end{lemma}
\begin{proof}
We use an argument modeled on the proof of \cite[Theorem 3.1.3]{kedlaya-relative}.
Define the rings $\calS^{\inte},\calS$ as in Lemma~\ref{L:integrality}.
Let $\iota_1, \iota_2: \tilde{\calR}_A \to \calS$ be the homomorphisms
$\iota_1(x) = x \otimes 1, \iota_2(x) = 1 \otimes x$.
Choose a basis $\be_1,\dots,\be_d$ of $\tilde{M}_0$. 
We then have $\varphi(\be_j) = \sum_i F_{ij} \be_i$ for some invertible matrix $F$ over $\tilde{\calR}^{\inte}_A$,
and $\iota_1(\be_j) = \sum_i U_{ij} \iota_2(\be_i)$ for some invertible matrix $U$ over $\calS$.
Applying $\varphi$, we find that $U \iota_1(F) = \iota_2(F) U$. By Lemma~\ref{L:integrality},
$U$ has entries in $\calS^{\inte}[\varpi^{-1}]$.

Take the inclusion $M \to \tilde{M}$, take completed tensor products
over $\calR_A$ with $\tilde{\calR}_A$, then identify the source with $\tilde{M}$ to obtain
a map $f: \tilde{M} \to \tilde{M} \widehat{\otimes}_{\calR_A} \tilde{\calR}_A$.
By the previous paragraph, $f(\tilde{M}_0)$ is contained in $\tilde{M}_0 \widehat{\otimes}_{\calR_A^{\inte}}
\tilde{\calR}_A^{\inte}[\varpi^{-1}]$. Since the map $\calR_A \to \tilde{\calR}_A$ admits an $\calR_A$-linear section
by Lemma~\ref{L:Robba split}, it follows that if we tensor the decomposition \eqref{eq:splitting S} of $\calS$ with $M$ and then decompose $\bv \in \tilde{M}_0$,
the resulting components in $M$ (tensored with $M_{n_1} \otimes M_{n_2}$) in fact belong to $M \cap
\tilde{M}_0[\varpi^{-1}] = M_0[\varpi^{-1}]$. More precisely, they all belong to $\varpi^{-j} M_0$ for some
nonnegative integer $j$ which may be chosen independent of $\bv$.
This implies that $\bigcup_{n=0}^\infty \varphi^{-n}(M_0)$ is dense in $\varpi^k \tilde{M}_0$;
since $\bigcup_{n=0}^\infty \varphi^{-n}(M_0)$ is $\varpi$-saturated in $\tilde{M}_0$,
we may also deduce the original claim.
\end{proof}

\begin{lemma} \label{L:etale descent}
Let $M$ be a projective $\varphi$-module over $\calR_A$
and put $\tilde{M} = M \otimes_{\calR_A} \tilde{\calR}_A$.
Let $\tilde{M}_0$ be an \'etale model of $\tilde{M}$.
Then the intersection $M_0 = M \cap \tilde{M}_0$ within $\tilde{M}$ is an \'etale model of $M$.
Moreover, if $\tilde{M}_0$ is free as a module over $\tilde{\calR}^{\inte}_A$, then $M_0$ is free as a module
over $\calR_A^{\inte}$.
\end{lemma}
\begin{proof}
We first treat the case where $\tilde{M}_0$ is free.
Choose a basis $\be_1,\dots,\be_d$ of $\tilde{M}_0$ and pick any $r>0$.
By Lemma~\ref{L:density}, we can
find $\be'_1,\dots,\be'_d \in \bigcup_{n=0}^\infty \varphi^{-n}(M_0)$ so that the matrix $U$ over 
$\tilde{\calR}_A$ for which $\be'_j = \sum_i U_{ij} \be_i$ has entries in
$\tilde{\calR}_A$ and satisfies $\left|U-1\right|_s < 1$ for $s \in [r, qr]$.
On the other hand, $U$ also has entries with nonnegative $\varpi$-adic valuation,
so $U$ is forced to be invertible.
It follows that $\be'_1,\dots,\be'_d$ also form a basis of $\tilde{M}_0$. Choose $n$ so that 
$\be'_1,\dots,\be'_d \in \varphi^{-n}(M_0)$; then $\varphi^n(\be'_1), \dots, \varphi^n(\be'_d)$ form
a basis of $M_0$, so $M_0$ is an \'etale model of $M$.

We next reduce the general case to the case of a free \'etale model.
By Lemma~\ref{L:descend projective} plus the isomorphism $\varphi^* \tilde{M}_0 \cong \tilde{M}_0$, we can find a finite projective module
$M'_0$ over $\calR^{\inte}_A$ and an isomorphism
$M'_0 \otimes_{\calR^{\inte}_A} \tilde{\calR}_A \cong \tilde{M}_0$
of modules over $\tilde{\calR}_A$.
We can then choose a finite projective module $N_0$ over $\calR_A^{\inte}$
such that $F_0 = M'_0 \oplus N_0$ is free over $\calR_A$. 
Put $N = N_0 \otimes_{\calR_A^{\inte}} \calR_A$ and $F = F_0 \otimes_{\calR_A^{\inte}} \calR_A$.
We may equip $F$ with the structure of an \'etale $\varphi$-module over $\calR_A$
admitting $F_0$ as an \'etale model 
by decreeing that $\varphi$ fixes some basis of $F_0$.
We may then equip $N$ with the structure of an \'etale $\varphi$-module over $\calR_A$
admitting $N_0$ as an \'etale model: define the action of $\varphi$ on $N_0$
by mapping $N_0$ into $F_0$, applying $\varphi$ on $F_0$ as described, then projecting $F_0$ to $N_0$.
Put $\tilde{N} = N \otimes_{\calR_A} \tilde{\calR}_A$ and $\tilde{N}_0 = N_0 \otimes_{\calR_A^{\inte}}
\tilde{\calR}_A^{\inte}$; by applying the previous paragraph,
we deduce that the intersection $(M \oplus N) \cap (\tilde{M}_0 \oplus \tilde{N}_0) = M_0 \oplus N_0$ within
$\tilde{M} \oplus \tilde{N}$ is an \'etale model of $M \oplus N$. 
Consequently, $M_0$ is an \'etale model of $M$.
\end{proof}

\begin{theorem} \label{T:etale after base change}
In the following diagram, each solid arrow has the following property: a projective $\varphi$-module over the source
ring is \'etale if and only if its base extension along the arrow is \'etale.
\begin{equation} \label{eq:ring diagram3}
\xymatrix@R10pt@C10pt{
\calR^{\bd}_A \ar@{-->}[rr] \ar[dd]\ar[rd] && \calR_A \ar[rd]\\
& \breve{\calR}^{\bd}_A \ar@{-->}[rr] \ar[dd] \ar[rd] && \breve{\calR}_A \ar[rd] \\
\calE_A \ar[rd] && \tilde{\calR}^{\bd}_A \ar@{-->}[rr] \ar[dd] && \tilde{\calR}_A \\
& \breve{\calE}_A \ar[rd] \\
&& \tilde{\calE}_A
}
\end{equation}
\end{theorem}
\begin{proof}
We handle the various cases as follows.

\begin{itemize}
\item
For the morphisms $\calR_A^{\bd} \to \breve{\calR}_A^{\bd}$, $\calR_A \to \breve{\calR}_A$,
$\calE_A \to \breve{\calE}_A$, see Remark~\ref{R:reflect etale}.

\item
For the morphism $\calR_A^{\bd} \to \calE_A$, apply Lemma~\ref{L:restrict etale to R};
this also resolves the case  $\breve{\calR}_A^{\bd} \to \breve{\calE}_A$,
and a similar argument applies to $\tilde{\calR}_A^{\bd} \to \tilde{\calE}_A$.
\item
For the morphism $\breve{\calE}_A \to \tilde{\calE}_A$,
we may instead consider the morphism $\calE_A \to \tilde{\calE}_A$.
In the notation of Remark~\ref{R:reflect etale}, Lemma~\ref{L:base extension E}
implies that the isomorphism $M \otimes_{R_1} R_2 \to N \otimes_{R_1} R_2$ descends to an isomorphism $M \to N$.

\item
For the morphism $\breve{\calR}_A^{\bd} \to \tilde{\calR}_A^{\bd}$, we may formally deduce from the previous cases.

\item
For the morphism $\breve{\calR}_A \to\tilde{\calR}_A$, apply 
Lemma~\ref{L:etale descent}.
\end{itemize}
\end{proof}

\begin{remark} \label{R:etale after base change}
For the morphisms $\calR_A^{\bd} \to \calR_A$, $\breve{\calR}_A^{\bd} \to \breve{\calR}_A$, $\tilde{\calR}_A^{\bd} \to \tilde{\calR}_A$, some examples of non-\'etale $\varphi$-modules with \'etale base extensions arise from the fact that a family of elliptic curves which is generically ordinary can specialize to a supersingular elliptic curve.
\end{remark}

\section{Enhanced $\varphi$-modules}

In general, the base extension functor in Lemma~\ref{L:etale fully faithful}(a) is not essentially surjective. However, many natural examples arise in conjunction with some additional structure, and in some such settings one can establish essential surjectivity for $\varphi$-modules. We next introduce a flexible framework for considering $\varphi$-modules equipped with additional structure.

\begin{defn} \label{D:monoid ring}
Let $M$ be a monoid equipped with a left action on a ring $R$. The \emph{twisted monoid ring} $R\{M\}$ is the free left $R$-module on the generating set $\{[m]: m \in M\}$, equipped with the multiplication
\[
\left( \sum_{m \in M} r_{1,m} [m] \right) \left( \sum_{m \in M} r_{2,m} [m] \right)
= \sum_{m \in M} \left( \sum_{m_1, m_2 \in M: m_1 m_2 = m} r_{1,m_1} m_1(r_{2,m_2}) \right)[m].
\]
This defines an associative, but typically not commutative, $R$-algebra.
When the monoid $M$ is free on a single generator $m$, we will write $R\{m\}$ instead of $R\{M\}$; this gives a twisted polynomial ring in the sense of Ore \cite{ore}.
\end{defn}

\begin{defn}
Let $*$ be one of the symbols $\emptyset, \inte, \bd$.
By an \emph{enhanced $q$-power Frobenius lift} on $\calR^*_A$, we will mean a (not necessarily commutative) ring $R$ equipped with a homomorphism $\calR^{*}_A\{\varphi\} \to R$,
where $\varphi$ is a relative $q$-power Frobenius lift on $\calR^*_A$.
Given such an object, 
an \emph{enhanced projective $\varphi$-module}
(resp.\ an \emph{enhanced coherent $\varphi$-module}) over $R$
is a left $R$-module which is a projective 
(resp. coherent) $\varphi$-module over $\calR^{*}_A$.
We may write \emph{$R$-enhanced} instead of \emph{enhanced} if we need to specify $R$.
\end{defn}

\begin{remark} \label{R:p-adic Lie}
A typical example of an enhanced $q$-power Frobenius lift is the twisted monoid ring
corresponding to the product of the free monoid generated by $\varphi$ with a group $\Gamma$ acting on $\calR_A$; for such data, an enhanced $\varphi$-module is simply a $\varphi$-module equipped with a compatible $\Gamma$-action. 
However, in applications, one typically has a topology on $\Gamma$ 
for which the action on $\calR_A$ is continuous (that is, the action map $\Gamma \times \calR_A \to \calR_A$ is continuous) and we want the action on $\varphi$-modules to be continuous (in the corresponding sense); this amounts to replacing the monoid algebra with a suitable completion thereof. In most cases of interest, the mechanism
for forming this completion is addressed by the considerations of \S\ref{sec:local analyticity}.
\end{remark}

\begin{defn}
We say that an enhanced $q$-power Frobenius lift admits \emph{overconvergent descent} if the base extension functor from enhanced $\varphi$-modules over $\calR_A^{\inte}$ to enhanced $\varphi$-modules over $\calE_A^{\inte}$
(which is fully faithful by Lemma~\ref{L:etale fully faithful}(a)) is an equivalence of categories.
\end{defn}

\begin{defn}
There are two different ways to define the notion of an \emph{\'etale enhanced $\varphi$-module} over some base ring $R$.
The most obvious one is to take an enhanced $\varphi$-module over $R_0$ where $R_0$ is obtained from $R$ as in
Definition~\ref{D:etale phi-module}. The other is to take an enhanced $\varphi$-module over $R$ whose underlying $\varphi$-module admits an \'etale model (but this model need not inherit ``enhanced'' structure); one could refer to the latter as a \emph{weakly \'etale enhanced $\varphi$-module}, but we will not see any such objects in the remainder of this paper.
\end{defn}

\section{Local analyticity}
\label{sec:local analyticity}

We will mostly apply Remark~\ref{R:p-adic Lie} with $\Gamma$ being a $p$-adic Lie group
satisfying a fairly mild continuity condition; however, this will formally promote to a much stronger property on the action.
\begin{remark}
For $\Gamma$ a group acting on a ring $R$, we will have use of two identities akin to the Leibniz rule and chain rule for derivations: for $\gamma, \gamma_1, \gamma_2 \in \Gamma$
and $x,y,z \in R$,
\begin{align} \label{eq:Leibniz}
(\gamma-1)(yz) &= z (\gamma-1)(y) + \gamma(y) (\gamma-1)(z) \\
\label{eq:chain}
(\gamma_1 \gamma_2 - 1)(x) &= (\gamma_1-1)(\gamma_2(x)) + (\gamma_2-1)(x).
\end{align}
\end{remark}

\begin{lemma} \label{L:mod p analytic action}
Let $\Gamma$ be a $p$-adic Lie group equipped with an action on $A$ satisfying the following conditions.
\begin{enumerate}
\item[(i)] For each $\gamma \in \Gamma$, the map $\gamma: A \to A$ is continuous.
\item[(ii)] For each $\overline{x} \in A$, the map $\overline{x}: \Gamma \to A$ is continuous.
\end{enumerate}
Then the action map $\Gamma \times A \to A$ is continuous.
Moreover, if we equip $A$ with the quotient norm induced by some presentation,
then for each $c \in (0,1)$
there exists a pro-$p$ compact open subgroup $H$ of $\Gamma$
such that for all $\gamma \in H$ and all $\overline{x} \in A$,
\begin{equation} \label{eq:continuity}
\left| (\gamma-1)(\overline{x}) \right| < c \left| \overline{x}  \right|.
\end{equation}
\end{lemma}
\begin{proof}
Let $\overline{f}:k (( \pi )) \langle \overline{T}_1,\dots,\overline{T}_{m} \rangle \llangle \overline{U}_1,\dots,\overline{U}_{n} \rrangle \to A$
be the chosen presentation.
Since $\Gamma$ is a $p$-adic Lie group, 
there exist a pro-$p$ compact open subgroup $H_0$ of $\Gamma$ and elements $\gamma_1,\dots,\gamma_n \in H_0$ such that:
\begin{itemize}
\item
the map
\[
\ZZ_p^n \to H_0, \qquad (e_1,\dots,e_n) \mapsto \gamma_1^{e_1} \cdots \gamma_n^{e_n}
\]
is a homeomorphism;
\item
for each nonnegative integer $i$, the image $H_i$ of $p^i \ZZ_p^n$ in $H_0$ is a compact open subgroup of $\Gamma$;
\item
the groups $H_0, H_1, \dots$ form a neighborhood basis of the identity in $\Gamma$.
\end{itemize}
By condition (ii), for any given $\overline{x} \in A$, 
\eqref{eq:continuity} holds for $\gamma = \gamma_j^{p^i}$ for any sufficiently large $i$.
In particular, we may ensure that this holds for $\overline{x}$ in the image of $\{\overline{\pi}, \overline{\pi}^{-1}, \overline{T}_1,\dots,\overline{T}_m,
\overline{U}_1,\dots,\overline{U}_n\}$ and $\gamma \in \{\gamma_1^{p^i},\dots,\gamma_n^{p^i}\}$; using \eqref{eq:Leibniz},
for the same choices of $\gamma$
we get \eqref{eq:continuity} also for 
every $\overline{x} \in A$. By condition (i),
we may extend to all $\gamma \in H_i$.
\end{proof}

\begin{remark}
From the proof of Lemma~\ref{L:mod p analytic action}, we also see that we can either relax (i) by only considering $\gamma$ running over some set of topological generators of $\Gamma$, or relax (ii) by only considering $\overline{x}$ running over some set of topological $k$-algebra generators of $A$. However, we cannot make both relaxations at once.
\end{remark}

By a similar argument, we obtain the following result.
\begin{lemma} \label{L:lifted analytic action}
Let $\Gamma$ be a $p$-adic Lie group equipped with an action on $S$ satisfying the following conditions.
\begin{enumerate}
\item[(i)] For each $\gamma \in \Gamma$, the map $\gamma: S \to S$ is a dagger morphism.
\item[(ii)] For each $\overline{x} \in A$, the map $\overline{x}: \Gamma \to A$ is continuous.
\end{enumerate}
Then the action map $\Gamma \times S \to S$ is continuous both for the weak topology on $S$ and for the LF topology on $S$.
Moreover, if we fix a presentation on $S$, then for some $r_0 > 0$, for each $c \in (0,1)$, for each $r \in (0,r_0]$,
there exists a pro-$p$ compact open subgroup $H$ of $\Gamma$
such that for all $\gamma \in H$ and all $x \in S^r$,
\begin{equation} \label{eq:continuity2}
\left| (\gamma-1)(x) \right|_r < c \left| x  \right|_r.
\end{equation}
\end{lemma}
\begin{proof}
By conditions (i) and (ii), the conclusions of  Lemma~\ref{L:mod p analytic action} apply. 
Retain notation as in the proof of Lemma~\ref{L:mod p analytic action}
with $f: \bA\langle T_1,\dots,T_{m} \rangle\llangle U_1, \dots, U_{n} \rrangle^{\dagger} \to S$ being the fixed presentation of $S$.
(We will not use (ii) explicitly hereafter.)

By condition (i) and Lemma~\ref{L:mod p analytic action},
we may choose $H_0$ so that for each $j$, $\gamma_j: S \to S$ is a dagger morphism which is norm-scaling with parameter $1$.
By Lemma~\ref{L:dagger compatibility}, there exist $r_0 > 0, c_1 \geq 1$ such that for $0 < r \leq r_0$, $\gamma_j$ maps $S^r$ to $S^r$ and
\[
\left| \gamma_j(x) \right|_r \leq c_1^r \left| x \right|_r \qquad (x \in S^r).
\]
This implies that for each nonnegative integer $i$,
\[
\left| (\gamma_j^{p^i}-1)(x) \right|_r \leq c_1^r \left| x \right|_r \qquad (x \in S^r).
\]
However, for any given $c_2 \in (0,1)$, Lemma~\ref{L:mod p analytic action} implies that for $i$ sufficiently large
(independent of $x$ or $r$),
\[
\left| (\gamma_j^{p^i}-1)(\overline{x}) \right| < c \left| \overline{x} \right| \qquad (\overline{x} \in A).
\]
For any $x \in S^r$ with nonzero reduction $\overline{x} \in A$,
we may combine the two previous inequalities using Remark~\ref{R:Hadamard} to deduce that for any $c>0$, for $i$ sufficiently large (depending on $x$ in addition to $c$), \eqref{eq:continuity2} holds for
$\gamma = \gamma_j^{p^i}$ for this particular value of $x$. 
By applying this inequality for $x$ in the image of $\{\pi, \pi^{-1}, T_1,\dots,T_m,
U_1,\dots,U_n\}$ 
and invoking \eqref{eq:Leibniz}, we deduce that for $i$ sufficiently large, \eqref{eq:continuity2} holds
with $\gamma = \gamma_j^{p^i}$ for all $x \in S^r$.

We now prove the claim in the case where $c \geq p^{-1}$.
Fix $i$ as above and take $H$ to be the closure of the subgroup generated by $\gamma_1^{p^i},\dots,\gamma_j^{p^i}$.
To verify \eqref{eq:continuity2} for all $\gamma \in H$, by \eqref{eq:chain} it suffices to do this for
$\gamma = \gamma_1^{p^i e}$ with $e \in \ZZ_p$. Write $e = e_0 + p e_1$ with $e_0 \in \ZZ$; by \eqref{eq:chain} plus the previous argument, the claims for $e$ and $p e_1$ are equivalent. However,
\[
\gamma_j^{p e_1} - 1 = 
\sum_{h=1}^{p} \binom{p}{h} (\gamma_j^{e_1} - 1)^h
\]
and so
\[
\left| (\gamma_j^{pe_1} - 1)(x) \right|_r \leq \max_h\{\left| p (\gamma_j^{e_1}-1)(x) \right|_r\}
\leq p^{-1} \left| x \right|_r.
\]

By similar logic, the claim for any single value of $c \in (0,1)$ implies the same claim with $c$ replaced by
$\max\{c^p, p^{-1} c \}$. By iterating this, we deduce the claim for all $c \in (0,1)$.
\end{proof}

\begin{cor} \label{C:lifted analytic action}
With notation as in Lemma~\ref{L:lifted analytic action}, for each sufficiently small $r>0$, 
the action of $\Gamma$ is locally analytic (in the sense of Lazard \cite{lazard})
with respect to $\left| \bullet \right|_r$.
\end{cor}
\begin{proof}
Let $H$ be a subgroup of the form given by Lemma~\ref{L:lifted analytic action} for some $c \in (0, 1/p^{1/(p-1)})$;
more precisely, we can and will assume that there exist elements $\gamma_1,\dots,\gamma_n \in H$ such that
\[
\ZZ_p^n \to H, \qquad (e_1,\dots,e_n) \mapsto \gamma_1^{e_1} \cdots \gamma_n^{e_n}
\]
is a homeomorphism. We have
\begin{equation} \label{eq:analytic action}
\gamma_1^{e_1} \cdots \gamma_n^{e_n}(x) = \sum_{l_1,\dots,l_n=0}^\infty 
\left( \prod_{j=1}^n 
e_j(e_j-1)\cdots(e_j-l+1) \right) \frac{(\gamma_1-1)^{e_1}}{e_1!} \cdots \frac{(\gamma_n-1)^{e_n}}{e_n!}(x);
\end{equation}
more precisely, this holds \emph{a priori} when $e_1,\dots,e_n$ are positive integers (as then the sum is finite),
but the estimate
\[
\left| \frac{(\gamma_1-1)^{e_1}}{e_1!} \cdots \frac{(\gamma_n-1)^{e_n}}{e_n!} (x) \right|_r \leq c^{e_1+\cdots+e_n} p^{(e_1+\cdots+e_n)/(p-1)} \left| x \right|_r.
\]
implies that the series converges for all $e_1,\dots,e_n$, and continuity of the action
(Lemma~\ref{L:lifted analytic action}) then implies that \eqref{eq:analytic action} holds in all cases.
From this series representation, one reads off the claim.
\end{proof}

\begin{defn}
Let $\Gamma$ be a profinite group. 
The group algebra $\ZZ_p[\Gamma]$ admits an augmentation morphism $\ZZ_p[\Gamma] \to \ZZ_p$
taking $\sum_{m \in \Gamma} r_m[m]$ to $\sum_{m \in \Gamma} r_m$; the kernel of this morphism is the \emph{augmentation ideal} of $\ZZ_p[\Gamma]$. Define the completed group algebra
$\ZZ_p \llbracket \Gamma \rrbracket$ to be the completion of $\ZZ_p[\Gamma]$ with respect to the augmentation ideal.

Now let $\Gamma$ be a $p$-adic Lie group. We may then define $\ZZ_p \llbracket \Gamma \rrbracket$ as
the tensor product $\ZZ_p[\Gamma] \otimes_{\ZZ_p[H_0]} \ZZ_p \llbracket H_0 \rrbracket$ 
where $H_0$ is a compact open subgroup of $\Gamma$; it is easily verified that this definition does not depend on the choice of $H_0$.
\end{defn}

\begin{defn} \label{D:completed monoid ring}
Let $\Gamma$ be a $p$-adic Lie group. We say that an action of $\Gamma$ on $S$ is \emph{Lazardian} 
if it satisfies hypotheses (i) and (ii) of Lemma~\ref{L:lifted analytic action}. In this case,
Lemma~\ref{L:lifted analytic action} implies that $\Gamma$ acts continuously on $\calR_A$,
and
Corollary~\ref{C:lifted analytic action} implies that this action extends to an action of $\ZZ_p \llbracket \Gamma \rrbracket$. We may thus define a ring structure on the left $\calR_A$-module
\[
\calR_A \llbrace \Gamma \rrbrace := \calR_A \otimes_{\ZZ_p} \ZZ_p \llbracket \Gamma \rrbracket
\]
by extending the formula for multiplication in $\calR_A\{\Gamma\}$ (Definition~\ref{D:monoid ring}).
For $M$ a left $\calR_A \{ \Gamma \}$-module whose underlying $\calR_A$-module is finite projective, the action of $\Gamma$ on $M$ is continuous if and only if $M$ promotes to a left $\calR_A\llbrace \Gamma \rrbrace$-module.
\end{defn}

\section{Some examples}
\label{sec:examples}

We now make contact with the literature by interpreting a number of prior constructions as examples of the setup
we have introduced.
In all cases involving a group action, the completion of the monoid rings indicated in Remark~\ref{R:p-adic Lie} is the one described in
Definition~\ref{D:completed monoid ring}.

\begin{example} \label{exa:pDE}
For $A = k((\overline{\pi}))$, the ring $\calR_A$ 
is the ring of formal Laurent series $\sum_{i \in \ZZ} c_i \pi^i$ with coefficients in $K$ which converge on some open annulus with outer radius 1. 
This ring is not noetherian, but is an integral domain in which every finitely generated ideal is principal (i.e., a \emph{B\'ezout domain}).
The construction also makes sense when $K$ is not discretely valued, but we ignore this point here.

For $K$ of mixed characteristics, the ring $\calR_A$ coincides with the \emph{Robba ring over $K$} in the theory of $p$-adic differential equations
(e.g., \cite[Definition~15.1.4]{kedlaya-course}).
In this context, it is crucial that formal differentiation of series gives rise to a well-defined $K$-linear derivation $\frac{d}{d\pi}: \calR_A \to \calR_A$; a typical object of study is a differential module over $\calR_A$ whose underlying $\calR_A$-module is finite projective (or equivalently, finite free). Note that a differential module over $\calR_A$ can itself be viewed as a left module for the twisted polynomial ring $\calR_A\left\{\frac{d}{d\pi} \right\}$.

One can also consider a differential module over $\calR_A$ with \emph{Frobenius structure},
with respect to some relative $q$-power Frobenius lift $\varphi$, i.e., a differential module $M$ equipped with an isomorphism with its $\varphi$-pullback.
Such an object may be viewed as an enhanced projective $\varphi$-module for the noncommutative ring
\[
R = \calR_A \left\{ \varphi, \frac{d}{d\pi} \right\} / \left(
\frac{d}{d\pi} \varphi - \left(\frac{d\varphi(\pi)}{d\pi}\right) \varphi \frac{d}{d\pi} \right).
\]
An important theorem in the subject (called variously \emph{Crew's conjecture}, the \emph{$p$-adic Turrittin theorem}, or the \emph{$p$-adic local monodromy theorem})
asserts that (under some conditions on $\varphi$, e.g., when $\varphi$ is an absolute Frobenius lift)
for any enhanced projective $\varphi$-module $M$, there exists a 
finite \'etale $A$-algebra $B$ such that $M \otimes_{\calR_A} \calR_B$ is unipotent as a differential module over $\calR_B$,
i.e., it is isomorphic to a successive extension of copies of $\calR_B$ itself.
For further discussion, see for example \cite[Chapters~19--20]{kedlaya-course}.

This example does not admit overconvergent descent. For $\frako = W(k)$,
by a theorem of Crew \cite[Theorem~2.1]{crew-f}, the category of enhanced $\varphi$-modules over $\calE_A^{\inte}$ is equivalent to the category of continuous representations of $G_A$ on finite free $\frako$-modules.
By another theorem of Tsuzuki \cite[Theorem~4.2.6]{tsuzuki-finite}, the category of enhanced $\varphi$-modules over $\calR_A^{\inte}$ is equivalent to the subcategory of representations with finite monodromy (meaning that the image of the inertia subgroup $G_{\overline{k}((\overline{\pi}))}$ is finite).
\end{example}

\begin{remark} \label{R:pDE descent obstruction}
In Example~\ref{exa:pDE}, the categories of $\varphi$-modules and enhanced $\varphi$-modules over $\calE_A^{\inte}$ are equivalent; that is, any $\varphi$-action is compatible with a unique differential structure \cite[Theorem~3.3.2]{tsuzuki-over}. This is not true over $\calR_A^{\inte}$: while \textit{loc. cit.} implies that
the enhanced $\varphi$-modules over $\calR_A^{\inte}$ form a full subcategory of the $\varphi$-modules over $\calR_A^{\inte}$, one also finds therein an example of an enhanced $\varphi$-module over $\calE_A^{\inte}$ for which the underlying $\varphi$-module is realized over $\calR_A^{\inte}$, but the differential module is not. 
That is, there exists a $\varphi$-module $M$ over $\calR_A^{\inte}$ for which the unique compatible differential action on $M \otimes_{\calR_A^{\inte}} \calE_A^{\inte}$ does not preserve $M$.
\end{remark}

\begin{example} \label{exa:pHT}
Take $k = \FF_p$, $K = \QQ_p$, $A = \FF_p((\overline{\pi}))$. The resulting ring $\calR_A$ appears frequently in $p$-adic Hodge theory as the base ring $\bB^{\dagger}_{\mathrm{rig}, \QQ_p}$ of the theory of \emph{$(\varphi, \Gamma)$-modules}, described originally by Berger \cite{berger-inv}
as a modification of Fontaine's original version \cite{fontaine-phigamma}.
In this context, $\calR_A$ naturally occurs equipped with a particular Frobenius lift
\[
\varphi: \sum_i c_i \pi^i \mapsto \sum_i c_i ((1 + \pi)^p-1)^i
\]
as well as a continuous action (see Remark~\ref{R:p-adic Lie}) of the group $\Gamma = \ZZ_p^\times$ given by
\[
\gamma: \sum_i c_i \pi^i \mapsto \sum_i c_i ((1 + \pi)^\gamma - 1)^i \qquad (\gamma \in \Gamma),
\]
where $(1 + \pi)^\gamma$ is interpreted using the binomial expansion. 
The category of enhanced $\varphi$-modules over $\calE_A^{\inte}$ is equivalent to the category of continuous representations of $G_{\QQ_p}$ on finite free $\ZZ_p$-modules (we refer to this hereafter as \emph{Fontaine's equivalence}).
This example admits overconvergent descent holds by a theorem of Cherbonnier--Colmez \cite{cherbonnier-colmez};
see also \cite{kedlaya-newphigamma} for an alternate proof.

The coincidence of rings between this example and Example~\ref{exa:pDE} is exploited in \cite{berger-inv} to prove a conjecture of Fontaine: every $p$-adic Galois representation which is de Rham is potentially semistable. Roughly speaking, one uses differential modules to convert a $p$-adic Galois representation into an associated Weil-Deligne representation.
\end{example}

\begin{example} \label{exa:phi-tau}
Take $k$ finite, $A = k((\overline{\pi}))$, $\frako = W(k)$, but now equip $\calR_A$
with the Frobenius lift
\[
\varphi: \sum_i c_i \pi^i \mapsto \sum_i c_i \pi^{pi}.
\]
This example admits overconvergent descent by work of Gao--Liu \cite{gao-liu} and Gao--Poyeton \cite{gao-poyeton}.

This construction also appears naturally in $p$-adic Hodge theory, notably in Kisin's classification of crystalline representations \cite{kisin-crys}. An analogue of the theory of $(\varphi, \Gamma)$-modules in this setting was described by Caruso \cite{caruso}; the analogue of Fontaine's equivalence is an equivalence between the category of enhanced $\varphi$-modules over $\calE_A^{\inte}$ and the category of continuous representations of $G_{\QQ_p}$ on finite free $\ZZ_p$-modules. However, this is more subtle than in Example~\ref{exa:pHT} because the enhanced $\varphi$-module includes an additional structure that does not come from a group action commuting with $\varphi$. This distinction arises from the fact that Example~\ref{exa:pHT} is (in a sense to be described later) controlled by the $p$-cyclotomic tower of $\QQ_p$, with the group $\ZZ_p^\times$ manifesting as the Galois group of this tower, whereas this example is controlled by a non-Galois Kummer tower.
\end{example}

\begin{example} \label{exa:kisin-ren}
Take $A = k((\overline{\pi}))$.
Let $K$ be a finite totally ramified extension of $\Frac W(k)$.
Let $L$ be a finite extension of $\QQ_p$ within $K$ with residue field $\FF_q$. 
Choose a uniformizer $\varpi_L \in \frako_L$,
form the corresponding Lubin-Tate formal $\frako_L$-module,
and form the absolute $q$-power Frobenius lift
\[
\varphi: \pi \mapsto [\varpi_L](\pi).
\]
Let $\Gamma$ be the Galois group of the extension of $K$ generated by the $\varpi_L$-power torsion of the formal group; then $\Gamma$ is identified with an open subgroup of $\frako_L^\times$ via the cyclotomic character, 
and acts on $\calR_A$ commuting with $\varphi$ (again via formal multiplication).

This variant of Example~\ref{exa:pHT} was introduced by  Fourquaux \cite{fourquaux}, and treated in more detail by Kisin--Ren \cite{kisin-ren}
with an eye towards studying crystalline representations. 
The analogue of Fontaine's equivalence in this setting is due to Chiarellotto--Esposito \cite{chiarellotto-esposito}.
This example does not admit overconvergent descent by the results of Kisin--Ren (see also \cite{fourquaux-xie}):
the enhanced $\varphi$-modules over $\calR_A^{\inte}$ correspond to Galois representations obeying an additional analyticity property.
\end{example}

\begin{example} \label{exa:tensor product}
For $A = k((\overline{\pi})) \langle \overline{T}_1, \dots, \overline{T}_m \rangle$,
$\calR_A$ consists of formal sums
\[
\sum_{i_0 \in \ZZ} \sum_{i_1,\dots,i_m=0}^\infty c_{i_0,\dots,i_m} \pi^{i_0} T_1^{i_0} \cdots T_m^{i_m} \qquad (c_{i_0,\dots,i_m} \in K)
\]
which converge on the product (over $K$) of some open annulus with coordinate $\pi$ having outer radius 1 with the product of the closed unit discs with coordinates $T_1,\dots,T_m$.
That is, we have
\[
\calR_A \cong 
\calR_{k((\overline{\pi}))} \widehat{\otimes}_K K \langle T_1,\dots,T_m \rangle,
\]
using the LF topology on Robba rings to define the completed tensor product.
Similarly, for $A = 
k((\overline{\pi})) \langle \overline{T}_1, \dots, \overline{T}_m \rangle
\llangle \overline{U}_1, \dots, \overline{U}_n \rrangle$, we have
\[
\calR_A \cong \calR_{k((\overline{\pi}))} \widehat{\otimes}_K K \langle T_1,\dots,T_m \rangle \llangle U_1, \dots, U_m \rrangle.
\]
This is a special case of Example~\ref{exa:tensor product2}, which see for further discussion.
\end{example}

\begin{example} \label{exa:tensor product2}
Let $A$ be arbitrary, and fix a presentation
\[
k((\overline{\pi})) \langle \overline{T}_1, \dots, \overline{T}_m \rangle 
\llangle \overline{U}_1, \dots, \overline{U}_n \rrangle \to A
\]
with kernel $I$. 
Put
\[
A' = A \langle \overline{T}_1, \dots, \overline{T}_m, 
\overline{T}'_1,\dots, \overline{T}'_{m'}
\rrangle  \llangle \overline{U}_1, \dots, \overline{U}_n, \overline{U}'_1,\dots,\overline{U}'_{n'} \rrangle/(I).
\]
As in Example~\ref{exa:tensor product}, we have
\[
\calR_{A'} \cong 
\calR_A \widehat{\otimes}_K K \langle T'_1,\dots,T'_{m'} \rangle \llangle U'_1,\dots,U'_{n'} \rrangle.
\]
Now let $B$ be a semiaffinoid algebra over $K$, and choose a presentation
\[
K \langle T'_1,\dots,T'_{m'} \rangle \llangle U'_1,\dots,U'_{n'} \rrangle \to B.
\]
Let $J$ be the kernel of the map $\frako \langle T'_1,\dots,T'_{m'} \rangle \llangle U'_1,\dots,U'_{n'} \rrangle \to B$, let $\overline{J}$ be the image of $J$ in
$k((\overline{\pi}))\langle \overline{T}'_1,\dots,\overline{T}'_{m'} \rangle \llangle \overline{U}'_1,\dots,\overline{U}'_{n'} \rrangle$
(note that it consists of elements in which $\overline{\pi}$ does not appear)
and put $A'' = A'/(\overline{J})$; then
\[
\calR_{A''} \cong \calR_{A} \widehat{\otimes}_K B.
\]
Using such an identification, structures on $\calR_A$ (e.g., a relative Frobenius lift) may be extended $B$-linearly to $\calR_{A''}$.
This type of base extension occurs frequently in the study of \emph{arithmetic families} of $(\varphi, \Gamma)$-modules, as in \cite{berger-colmez}, \cite{dee}, \cite{hellmann-schraen}, \cite{kedlaya-liu-families}, \cite{kpx}.

In such examples, one starts with a relative Frobenius lift as in Example~\ref{exa:pHT}; it should be possible to 
establish overconvergent descent in this setting (e.g., by adapting the method of \cite{kedlaya-newphigamma}), but we are only aware of partial results in this direction
(e.g., see \cite{berger-colmez}, \cite{kedlaya-liu-families}).
One issue is that there is not a complete analogue of the Fontaine equivalence: the natural functor from Galois representations to enhanced $\varphi$-modules over $\calE_A^{\inte}$ in this setting is not essentially surjective
(see \emph{loc. cit.}).
\end{example}

\begin{example}
For $A = k((\overline{\pi})) \langle \overline{T}_1, \dots, \overline{T}_m \rangle
\llangle \overline{U}_1,\dots, \overline{U}_n \rrangle$, consider the absolute Frobenius $\varphi$ on $\calR_A$ given by
\[
\pi \mapsto (1 + \pi)^p - 1, \qquad T_i \mapsto T_i^p, \qquad U_j \mapsto U_j^p.
\]
The ring $\calR_A$ then admits an action of the group $\Gamma = \ZZ_p^\times \ltimes \ZZ_p^{m+n}$
commuting with $\varphi$, where $\ZZ_p^\times$ acts as in Example~\ref{exa:pHT} fixing the $T_i$ and $U_j$, and $(e_1,\dots,e_m,f_1,\dots,f_n) \in \ZZ_p^{m+n}$ acts via
\[
\pi \mapsto \pi, \qquad T_i \mapsto (1 + \pi)^{e_i} T_i, \qquad U_j \mapsto (1 + \pi)^{f_j} U_j.
\]
These operations persist upon replacing $A$ with a semiaffinoid localization or a finite \'etale extension.

The analogue of the Fontaine equivalence in this context involves
$p$-adic \'etale local systems on rigid analytic spaces, or equivalently $p$-adic representations of the \'etale fundamental groups of connected rigid analytic spaces
\cite{andreatta-brinon}, \cite{andreatta-iovita}, \cite{kedlaya-liu2}. This example admits overconvergent descent by the results of \cite{kedlaya-liu2}.
\end{example}

\begin{example}
Let $\overline{\pi}_1,\dots,\overline{\pi}_m$ denote the images of $\overline{T}_1,\dots,\overline{T}_m$ in 
\[
A = k((\overline{\pi}))\llangle \overline{T}_1, \dots, \overline{T}_m \rrangle/(\overline{\pi} - \overline{T}_1 \cdots \overline{T}_m)
\cong k \llbracket \overline{\pi}_1,\dots,\overline{\pi}_m \rrbracket[\overline{\pi}_1^{-1},\dots, \overline{\pi}_m^{-1}];
\]
note that $A$ is not an affinoid algebra if $m>1$.
Construct $\calR_A$ using the dagger lift
\[
S = \bA \langle T_1,\dots,T_m \rangle \llangle U_1,\dots,U_n \rrangle/(\pi - T_1 \cdots T_m)
\]
and let $\pi_1,\dots,\pi_m$ denote the images of $T_1,\dots,T_m$ in $\calR_A$.
For $i=1,\dots,m$, let $\varphi_i: \calR_A \to \calR_A$ be the substitution taking 
$\pi_i$ to $(1+\pi_i)^p-1$ and fixing $\pi_j$ for $j \neq i$. The composition
$\varphi$ of $\varphi_1,\dots,\varphi_m$ (in any order) is an absolute $p$-power Frobenius lift.
Define an action of $\Gamma = (\ZZ_p^\times)^m$ on $\calR_A$ in such a way that for $i=1,\dots,m$, the $i$-th copy of $\ZZ_p^\times$ acts on $\pi_i$ as in
Example~\ref{exa:pHT} and fixes $\pi_j$ for $j \neq i$; then $\Gamma$ commutes with all of the $\varphi_i$.
This example is considered by Z\'abr\'adi \cite{zabradi-phigamma2};
the analogue of the Fontaine equivalence involves
representations of powers of the absolute Galois group of a $p$-adic field.
This example admits overconvergent descent by the results of \cite{carter-kedlaya-zabradi}, \cite{pal-zabradi}.
\end{example}

\begin{example} \label{exa:Berger}
Let $F$ be the finite unramified extension of $\QQ_p$ of degree $h$, and fix a Lubin-Tate formal $\frako_F$-module. Put
\[
A = k((\overline{\pi}))\langle \overline{T}_1,\dots,\overline{T}_h \rangle
\]
and let $Y_1,\dots,Y_h$ be the elements
$\overline{\pi} \overline{T}_1, \overline{\pi}^p \overline{T}_2, \dots, \overline{\pi}^{p^{h-1}} \overline{T}_h$
in $\calR_A$.
Define the ring homomorphism $\varphi_p$ as the substitution
\[
\varphi_p(Y_1) = Y_1, \dots, \varphi_p(Y_{h-1}) = Y_h, \varphi_p(Y_h) = [p](Y_1),
\]
where $[p]$ denotes the multiplication-by-$p$ operation in the chosen formal $\frako_F$-module.
Define an action of $\Gamma = \frako_F^\times$ by the formula
\[
\gamma(Y_i) = [\sigma^{i-1}(\gamma)](Y_i) \qquad (\gamma \in \frako_F^\times)
\]
where $\sigma: F \to F$ denotes the Frobenius automorphism; note that $\varphi_p$ does not commute with $\Gamma$ but $\varphi_p^h$ does. This examples was introduced by Berger \cite{berger-multi} to address the failure of overconvergent descent in the Fourquaux--Kisin--Ren construction.
It is not clear whether this example can be modified (by modifying the exact collection of structures included in the enhanced Frobenius lift) in such a way as to admit overconvergent descent; we prove a weaker statement
in this direction in Theorem~\ref{T:berger}.
\end{example}

\begin{remark}
As noted in the introduction, Z\'abr\'adi \cite{zabradi-phigamma} has described \emph{noncommutative Robba rings}; these have underlying topological groups coincident with examples we have already considered, but with a more exotic multiplication structure. It is not clear to us how to expand our framework so as to include these examples while still retaining the ability to prove meaningful statements; doing so might be significant for applications towards the $p$-adic Langlands correspondence.
\end{remark}

\section{APF extensions and perfectoid fields}
\label{sec:APF}

In preparation for the next section, we recall the relationship between the theory of \emph{arithmetically profinite extensions} of local fields of Fontaine--Wintenberger \cite{fontaine-wintenberger} and the theory of \emph{perfectoid fields} introduced in \cite{kedlaya-liu1} and \cite{scholze1} (and further discussed in \cite{kedlaya-newphigamma}).

\begin{defn}
Suppose that $M/L/K$ are extensions of algebraic extensions of $\QQ_p$ 
with $M/K$ Galois, and put $G = \Gal(M/K)$ and $H = \Gal(M/L)$.
For $i \geq -1$, let $G^i$ be the $i$-th ramification subgroup of $G$ in the upper numbering.
We say that $i$ is a \emph{ramification break} of $L/K$ if $G^i H \neq G^{i+\epsilon} H$ for all $\epsilon > 0$.
We denote by $i(L/K)$ the upper bound of those $i \geq -1$ for which $G^i H = G$.
Note that none of these definitions depends on the choice of $M$.
\end{defn}

\begin{hypothesis} \label{H:APF}
For the remainder of  \S\ref{sec:APF},
let $K$ be a finite extension of $\QQ_p$, let $K_\infty$ be an infinite algebraic extension of $K$,
and let $L_\infty$ be an infinite Galois extension of $K$ containing $K_\infty$. Put $G = \Gal(L_\infty/K)$ and
$H = \Gal(L_\infty/K_\infty)$.
\end{hypothesis}

\begin{defn} \label{D:APF}
We say that $K_\infty/K$ is \emph{arithmetically profinite (APF)}
if $G^i H$ is open in $G$ for all $i$. Equivalently, the ramification breaks of $K_\infty/K$ form an increasing sequence $i_0 < i_1 < \cdots$ of rational numbers and the fields
\[
K_j = K_\infty \cap L_\infty^{G^{i_j}} \qquad (j=0,1,\dots)
\]
(whose union is evidently equal to $K_\infty$) are \emph{finite}
extensions of $K$.
Note that any APF extension is \emph{almost totally ramified} and \emph{almost totally wildly ramified}.

We say that $K_\infty$ is \emph{strictly APF}
if $K_\infty$ is APF and
\[
\liminf_{j \to \infty} \frac{i'_j}{[K_j:K]} > 0,
\]
where 
\[
i'_j = \psi_{K_\infty/K}(i_j) = \int_0^{i_j} [G^0:G^v H^0] \,dv
\]
is the index in the lower numbering corresponding to $i_j$ in the upper numbering. Equivalently, we must have $c(K_\infty/K) > 0$ for
\[
c(K_\infty/K) = \liminf_{j \to \infty} \frac{p-1}{p} \frac{i(K_{j+1}/K_j)}{[K_{j+1}:K_0]}.
\]
\end{defn}

\begin{lemma} \label{L:APF to perfectoid}
Suppose that $K_\infty$ is strictly APF.
Let $v$ be the valuation on $K_\infty$ normalized so that $v(K^*) = \ZZ$. For $j >0$, we have the following.
\begin{enumerate}
\item[(a)]
For $y \in K_{j+1}$, we have 
\[
v(\Norm_{K_{j+1}/K_j}(y) - y^{[K_{j+1}:K_j]}) \geq \frac{p}{p-1} c(K_\infty/K).
\]
\item[(b)]
For $x \in K_j$, there exists $y \in K_{j+1}$ with
\[
v(\Norm_{K_{j+1}/K_j}(y) - x) \geq c(K_\infty/K).
\]
\end{enumerate}
\end{lemma}
\begin{proof}
Part (a) is straightforward. Part (b) is \cite[Lemme~3.1]{fontaine-wintenberger}.
\end{proof}
We can now make contact with a more modern point of view.

\begin{prop} \label{P:perfectoid}
If $K_\infty$ is strictly APF, then its completion is a \emph{perfectoid field} in the sense of
\cite[Definition~3.5.1]{kedlaya-liu1} or \cite[Definition~3.1]{scholze1}.
\end{prop}
\begin{proof}
By definition, $K_\infty$ is nondiscrete; by
Lemma~\ref{L:APF to perfectoid}, there exists $t \in \frakm_{K_{\infty}}$ such that the Frobenius endomorphism of $\frako_{K_\infty}/(t)$ is surjective. 
This completes the proof.
\end{proof}

\begin{defn}
Suppose that $K_\infty$ is strictly APF. 
By Proposition~\ref{P:perfectoid}, we may apply the \emph{perfectoid correspondence} 
\cite[Theorem~3.5.3]{kedlaya-liu1}
(called \emph{tilting} in \cite[\S 3]{scholze1}) to the completion of $K_\infty$ to produce a perfect field $\tilde{\bE}_{K_\infty}$ of characteristic $p$ which is complete with respect to a distinguished norm. The valuation subring of $\tilde{\bE}_{K_\infty}$ is naturally identified with the inverse limit of $\frako_{K_\infty}/(p)$ under Frobenius.
We call $\tilde{\bE}_{K_\infty}$ the \emph{perfect norm field} associated to $K_\infty$.

Let $\bE_{K_\infty}$ be the inverse limit of the $K_j$ under the norm maps. This set inherits an obvious multiplication map; it also admits a natural addition map under which 
\[
(x_j)_{j=0}^\infty + (y_j)_{j=0}^\infty = \left( \lim_{j' \to \infty} \Norm_{K_{j'}/K_j} (x_{j'} + y_{j'}) \right)_{j=0}^\infty.
\]
It is shown in \cite{fontaine-wintenberger} that $\bE_{K_\infty}$ is a local field of characteristic $p$ whose valuation subring is the inverse limit of the $\frako_{K_j}$ under the norm maps. Using Lemma~\ref{L:APF to perfectoid}, we may identify that valuation subring with the inverse limit of the rings $\frako_{K_j}/(t)$ for some $t \in \frakm_{K_\infty}$ under the map induced by both $\Norm_{K_{j+1}/K_j}$
and $x \mapsto x^{[K_{j+1}:K_j]}$. In particular, we may identify $\tilde{\bE}_{K_\infty}$ with the completed perfect closure of $\bE_{K_\infty}$. We call $\bE_{K_\infty}$ the \emph{imperfect norm field} associated to $K_\infty$.
\end{defn}

\section{A conjecture on locally analytic vectors}
\label{sec:locally analytic conj}

We formulate a conjecture closely related to the overconvergence theorem of Cherbonnier and Colmez.

\begin{hypothesis} 
Throughout \S\ref{sec:locally analytic conj},
assume Hypothesis~\ref{H:APF}, but assume further that $G$ is a $p$-adic Lie group and that $L_\infty$ is almost totally ramified.
Also fix some $r>0$.
\end{hypothesis}

\begin{prop} \label{P:Sen}
The fields $K_\infty$ and $L_\infty$ are strictly APF, 
so (by Proposition~\ref{P:perfectoid}) their completions are perfectoid.
\end{prop}
\begin{proof}
This is a theorem of Sen \cite{sen}.
\end{proof}

\begin{defn}
Let $\CC_p$ be a completed algebraic closure of $K$ containing $L_\infty$. By Proposition~\ref{P:Sen} and the compatibility of the perfectoid correspondence with algebraic field extensions
\cite[Theorem~3.5.6]{kedlaya-liu1}, \cite[Theorem~3.7]{scholze1},
$\CC_p$ corresponds to a perfect field $\tilde{\bE}$ of characteristic $p$ on which $G_{K_\infty}$ acts.
\end{defn}

\begin{lemma}
We have
\[
\tilde{\bE}^{K_\infty} = \tilde{\bE}_{K_\infty}, \qquad
\tilde{\bE}^{L_\infty} = \tilde{\bE}_{L_\infty}.
\]
\end{lemma}
\begin{proof}
For each finite extension $M$ of $K_\infty$ within $\CC_p$, the
perfectoid correspondence assigns to $M$ a finite extension $\tilde{\bE}_M$ of $\tilde{\bE}_{K_\infty}$ within $\tilde{\bE}$.
We then have an exact sequence
\[
0 \to \tilde{\bE}_{K_\infty} \to \tilde{\bE}_M \to \tilde{\bE}_M \otimes_{\tilde{\bE}_{K_\infty}} \tilde{\bE}_M
\]
where the last map is $x \mapsto x \otimes 1 - 1 \otimes x$.
This sequence is moreover \emph{strict} exact for the tensor product norm on the last factor, but since Frobenius acts bijectively on all terms of the sequence that coincides with the spectral seminorm.
The latter in turn equals the supremum over the factors of $\tilde{\bE}_M \otimes_{\tilde{\bE}_{K_\infty}} \tilde{\bE}_M$ 
\cite[Theorem~2.3.10]{kedlaya-liu1}.
We may then take the direct limit over $M$ and complete to obtain
a strict exact sequence
\[
0 \to \tilde{\bE}_{K_\infty} \to \tilde{\bE} \to \tilde{\bE} \widehat{\otimes}_{\tilde{\bE}_{K_\infty}} \tilde{\bE}
\]
where again the tensor product seminorm coincides with the supremum over the factors of $\tilde{\bE} \widehat{\otimes}_{\tilde{\bE}_{K_\infty}} \tilde{\bE}$
by \cite[Theorem~2.3.10]{kedlaya-liu1}. Since those factors correspond to elements of $G_{K_\infty}$, any element of $\tilde{\bE}^{G_{K_\infty}}$ must map to zero in $\tilde{\bE} \widehat{\otimes}_{\tilde{\bE}_{K_\infty}} \tilde{\bE}$
and hence must lie in $\tilde{\bE}_{K_\infty}$. This proves the first inequality; taking $K_\infty = L_\infty$ yields the second inequality.
\end{proof}

\begin{defn}
Put 
\[
\tilde{\bA} = W(\CC_p), \qquad \tilde{\bA}_{K_\infty} = W(\tilde{\bE}_{K_\infty}), \qquad \tilde{\bA}_{L_\infty}= W(\tilde{\bE}_{L_\infty}).
\]
Let $\tilde{\bA}^{\dagger,r}$ be the subring of $\tilde{\bA}$ consisting of $r$-overconvergent Witt vectors as in \cite[Lemma~1.7.2]{kedlaya-newphigamma}; this ring carries the norm $\left| \bullet \right|_r$ defined by
\[
\left| \sum_{n=0}^\infty p^n [\overline{x}_n] \right|_r = \max_n \{p^{-n} \left| \overline{x}_n \right|^r \}.
\]
Put $\tilde{\bA}^{\dagger,r}_* = \tilde{\bA}^{\dagger,r} \cap \tilde{\bA}_*$.
\end{defn}

\begin{defn} \label{D:uniformly analytic}
Let $M$ be a finite projective module over $\tilde{\calR}^r_{\tilde{\bE}}$ equipped with a continuous $G$-action. Choose a presentation of $M$ as a direct summand of a finite free module, then use this presentation to define norms $\left| \bullet \right|_s$ for $s \in (0,r]$ as in Remark~\ref{R:projective module topology}. We say that a subset $S$ of $M$ is \emph{uniformly $r$-analytic} if for any $c \in (0,1)$, there exists a pro-$p$ compact open subgroup $H$ of $G$ such that for all $\gamma \in H$, all $\bv \in S$, all $s \in (0,r]$, and all nonnegative integers $m$,
\begin{equation} \label{eq:uniformly analytic}
\left| (\gamma^{p^m}-1)(\bv) \right|_s < \max\{p^{i-m} c^{p^{i}s}: i=0,\dots,m\} \left| \bv \right|_s.
\end{equation}
Note that it is equivalent to impose this inequality with an extra constant factor on the right-hand side, as this can be absorbed by rechoosing $H$.
From this, we see that this condition does not depend on the choice of the presentation of $M$.
\end{defn}

The terminology \emph{uniformly $r$-analytic} is motivated by the following observation.
\begin{lemma} \label{L:analytic to analytic}
Retain notation as in Definition~\ref{D:uniformly analytic}. If $S$ is a $G$-stable subgroup of $M$ which is uniformly $r$-analytic, then for $H$ as in Definition~\ref{D:uniformly analytic},
for any $s \in (0,r]$ the function $S \times \log(H) \to M$ given by
$(\bv, t) \mapsto \exp(t)(\bv)$ is analytic in $\log(H)$, i.e., it is  computed by a power series which uniformly over $S$ is convergent with respect to the norm $\left| \bullet \right|_s$.
\end{lemma}
\begin{proof}
When $\dim_{\QQ_p} G = 1$, we may deduce the claim from the binomial expansion
\[
\gamma^n = \sum_{i=0}^\infty \binom{n}{i} (\gamma-1)^i;
\]
the general case then follows using the Campbell-Hausdorff formula.
\end{proof}

\begin{example}
For any nonnegative integer $n$ and any $\overline{x} \in \varphi^{-n}(\bE_{K_\infty})$,
for $\gamma \in G^{i_j}$ we have $\left| (\gamma-1)(\overline{x}) \right| \leq c^{p^{j}} \left| \overline{x} \right|$ for some $c \in (0,1)$ independent of $j$.
It follows that the Teichm\"uller lift $[\overline{x}]$ is uniformly $r$-analytic for all $r>0$.
However, for $\overline{y} \in \varphi^{-n}(\bE_{K_\infty})$, the set $\{[\overline{x}], [\overline{y}]\}$ is not
guaranteed to be uniformly $r$-analytic.
\end{example}

\begin{defn} \label{D:affinoid model}
For $x \in \tilde{\bA}^{\dagger,r}_{K_\infty}$, an \emph{affinoid model} of $x$ is a triple $(A,\alpha,y)$ where $A$ is a smooth affinoid algebra over $\FF_p((\overline{\pi}))$ admitting a continuous $G$-action, $\alpha: A \to \tilde{\bE}_{K_\infty}$ is a bounded $G$-equivariant homomorphism, and $y \in \calR^{\inte}_A$ is an element mapping to $x$ under the map $\calR^r_A \to \tilde{\bA}^{\dagger,r}_{K_\infty}$ induced by $\alpha$.
We say that $x$ is an \emph{affinoid element} of $\tilde{\bA}^{\dagger,r}_{K_\infty}$ if it admits an affinoid model.
Let $\tilde{\bA}^{\dagger,r,\aff}_{K_\infty}$ be the subring of affinoid elements of $\tilde{\bA}^{\dagger,r}_{K_\infty}$.

We say that an element $x \in \tilde{\bA}^{\dagger,r}_{K_\infty}$ is \emph{locally analytic} (for the action of $G$)
if there exists a pro-$p$ compact open subgroup $H$ of $G$ such that the map $\log(H) \to \tilde{\calR}^r_{\tilde{\bE}}$ given by $t \mapsto \exp(t)(x)$ is analytic in $\log(H)$. Let $\tilde{\bA}^{\dagger,r,\an}_{K_\infty}$ be the subring of locally analytic elements of $\tilde{\bA}^{\dagger,r}_{K_\infty}$.
By  Corollary~\ref{C:lifted analytic action} and Lemma~\ref{L:analytic to analytic}, we have
$\tilde{\bA}^{\dagger,r,\aff}_{K_\infty} \subseteq \tilde{\bA}^{\dagger,r,\an}_{K_\infty}$.
\end{defn}

\begin{conj} \label{C:locally analytic1}
We have 
$\tilde{\bA}^{\dagger,r,\aff}_{K_\infty} = \tilde{\bA}^{\dagger,r,\an}_{K_\infty}$.
\end{conj}

\begin{remark}
The definition of locally analytic vectors above (and in Definition~\ref{D:locally analytic vectors} below)
is in the style of Schneider--Teitelbaum \cite{schneider-teitelbaum} and Emerton \cite{emerton}.
One significant difficulty  with Conjecture~\ref{C:locally analytic1} 
is the lack of an obvious construction of $G$-stable affinoid models.
One case where this does not arise is when $K_\infty$ is the division field of a formal group, in which case the formal group law can be used to construct such models (as in \cite{berger-multi}) and then prove Conjecture~\ref{C:locally analytic1}
(as in \cite{berger-multi2}).
It may be possible to extend this construction to some additional cases in which $K_\infty$ arises by adjoining iterated inverse images of a finite rational map (e.g., the preperiodic points of an arithmetic dynamical system); the one-dimensional case of this has been treated by Cais--Davis \cite{cais-davis}.
\end{remark}

\begin{defn} \label{D:locally analytic vectors}
For $T$ a finite free $\ZZ_p$-module equipped with a continuous $G_K$-action, for $r>0$ put
\[
D^r_{K_\infty}(T) = (T \otimes_{\ZZ_p} \tilde{\bA}^{\dagger,r})^{G_{K_\infty}}.
\]
This is a module over $\tilde{\bA}^{\dagger,r}_{K_\infty}$,
and by the proof of \cite[Theorem~2.4.5]{kedlaya-newphigamma} (see also \cite[\S 8]{kedlaya-liu1}), the induced map
\[
D^r_{K_\infty}(T) \otimes_{\tilde{\bA}^{\dagger,r}_{K_\infty}} \tilde{\bA}^{\dagger,r}
\to T \otimes_{\ZZ_p} \tilde{\bA}^{\dagger,r}
\] 
is an isomorphism.

An element $\bv \in D^r_{K_\infty}(T)$ is \emph{locally analytic} (for the action of $G$)
if there exists a pro-$p$ compact open subgroup $H$ of $G$ such that the map $\log(H) \to T \otimes_{\ZZ_p} \tilde{\calR}^r_{\tilde{\bE}}$ given by $t \mapsto \exp(t)(\bv)$ is analytic in $\log(H)$.
Let $D^r_{K_\infty}(T)^{\an}$ be the 
set of locally analytic elements of $D^r_{K_\infty}(T)$ for the action of $G$;
 it is a module over $\tilde{\bA}^{\dagger,r,\an}_{K_\infty}$.
\end{defn}

\begin{conj} \label{C:locally analytic2}
Let $T$ be a finite free $\ZZ_p$-module equipped with a continuous $G_K$-action. Then for $r>0$, the natural map 
\[
D^r_{K_\infty}(T)^{\an} \otimes_{\tilde{\bA}^{\dagger,r,\an}_{K_\infty}} \tilde{\bA}^{\dagger,r}_{K_\infty}
\to D^r_{K_\infty}(T)
\]
is an isomorphism.
\end{conj}

\begin{remark}
One motivation for Conjecture~\ref{C:locally analytic2} is to approach a potential construction of the $p$-adic Langlands correspondence suggested by the work of Scholze--Weinstein \cite{scholze-weinstein} on moduli of $p$-divisible groups. Roughly speaking, given an $n$-dimensional representation $V$ of $G_K$, one expects the corresponding representation of $\GL_n(K)$ to be a sort of ``universal $(\varphi, \Gamma)$-module'' of $V$. That is, if we take $K_\infty$ to be the division field of a formal $\frako_K$-module of rank $n$, then its Galois group $G$ sits inside $\GL_n(K)$, and restricting the Langlands representation to $G$ should produce something related to the locally analytic representation coming from Conjecture~\ref{C:locally analytic2}. Moreover, this construction should in some sense be compatible with families; that is, it should work uniformly over the Scholze--Weinstein moduli space.
\end{remark}

\begin{remark}
Most of what we know about Corollary~\ref{C:locally analytic2} follows from cases of overconvergent descent
via the inclusion $\tilde{\bA}^{\dagger,r,\aff}_{K_\infty} \subseteq \tilde{\bA}^{\dagger,r,\an}_{K_\infty}$.
In particular, in the case where $K_\infty = L_\infty = K(\mu_{p^\infty})$, 
Conjecture~\ref{C:locally analytic2} follows immediately from 
overconvergent descent in Example~\ref{exa:pHT}, that is, from the 
theorem of Cherbonnier--Colmez \cite{cherbonnier-colmez}.
Similarly, in the case where $K_{\infty} = K(\varpi^{1/p^\infty})$ for some uniformizer $\varpi$ of $K$ and $L_{\infty}$ is the Galois closure of $K_\infty$, Conjecture~\ref{C:locally analytic2} follows from
overconvergent descent in Example~\ref{exa:phi-tau}, that is, from the theorem of
Gao--Poyeton \cite{gao-poyeton}.

An alternate proof of the Cherbonnier-Colmez theorem has been described in
\cite{kedlaya-newphigamma}. It is our hope that the ideas of that proof, implemented in some contexts
described in this paper, will lead to additional cases of Conjecture~\ref{C:locally analytic2}.
\end{remark}

\section{Analyticity in Berger's construction}
\label{sec:spreading}

To conclude, we link up with work of Berger \cite{berger-multi} to prove a new case of Conjecture~\ref{C:locally analytic2}. To do so, we relate our setup (in the case of an \emph{absolute} Frobenius lift)
to the theory of relative $\varphi$-modules developed in \cite{kedlaya-liu1}. 

\begin{hypothesis}
Throughout \S\ref{sec:spreading}, retain Hypothesis~\ref{H:absolute Frobenius lift},
but with the further restriction that $\varphi$ be an absolute $q$-power Frobenius lift.
Fix an element $\alpha$ of the Gel'fand spectrum $\calM(A)$, the space of bounded multiplicative seminorms on $A$ equipped with the evaluation topology. Note that $\alpha$ extends uniquely to $\tilde{A}$, yielding a natural homeomorphism $\calM(A) \cong \calM(\tilde{A})$.
There is also a homomorphism $\tilde{A} \to \calH(\alpha)$ to a complete nonarchimedean field $\calH(\alpha)$ such that the restriction of the norm on $\calH(\alpha)$ computes $\alpha$ on $\tilde{A}$.
\end{hypothesis}

\begin{defn}
A \emph{rational localization} of $A$ is a homomorphism $A \to B$ representing a rational subdomain of
$\calM(A)$ (e.g., in the sense of \cite[Definition~2.4.7]{kedlaya-liu1}); for such a homomorphism, the map $\calM(B) \to \calM(A)$ identifies $\calM(B)$ with a closed
subset of $\calM(A)$. We say that a rational localization $A \to B$ \emph{encircles} $\alpha$
if $\calM(B)$ is a neighborhood of $\alpha$ in $\calM(A)$; note that such neighborhoods form a neighborhood
basis of $\alpha$ in $\calM(A)$.
\end{defn}

\begin{theorem} \label{T:spread etale}
Let $M$ be a projective $\varphi$-module over $\calR_A$ such that 
$M \otimes_{\calR_A} \tilde{\calR}_{\calH(\alpha)}$ is 
\'etale. Then there exists a rational localization $A \to A'$ encircling $\alpha$ such that $M \otimes_{\calR_A}
\calR_{A'}$ admits a free \'etale model, and in particular is \'etale.
\end{theorem}
\begin{proof}
By \cite[Corollary~7.3.8]{kedlaya-liu1}, there exists a rational localization $\tilde{A} \to \tilde{A}'$ encircling $\alpha$
such that $M \otimes_{\calR_A} \tilde{\calR}_{\tilde{A}'}$ admits a free \'etale model.
By approximating the parameters defining the rational localization, we can write it as the base extension
of a rational localization of $\varphi^{-n}(A)$ for some nonnegative integer $n$; by raising the parameters
to the $q^n$-th power, we can then write $\tilde{A}'$ as the completed perfect closure of $A'$ for some rational localization $A \to A'$ encircling $\alpha$. By Lemma~\ref{L:etale descent}, $M \otimes_{\calR_A}
\calR_{A'}$ admits a free \'etale model.
\end{proof}

\begin{remark}
In general, $M \otimes_{\calR_A} \tilde{\calR}_{\calH(\alpha)}$ 
admits a \emph{slope polygon} according to a variant of the Dieudonn\'e-Manin classification; see \cite[\S 4]{kedlaya-liu1} for a summary. As a function of $\alpha$, the slope polygon is 
lower semicontinuous \cite[Theorem~7.4.5]{kedlaya-liu1}, which implies that \'etaleness at a point implies \'etaleness in an open neighborhood (as used in the proof of Theorem~\ref{T:spread etale}). 

The slope polygon need not be constant in a neighborhood of $\alpha$. If it is, then by \cite[Theorem~7.4.8]{kedlaya-liu1} there exists a rational localization $A \to A'$ encircling $\alpha$ such that $M \otimes_{\calR_A} \tilde{\calR}_{A'}$ admits a \emph{slope filtration} interpolating the Harder-Narasimhan filtration of $M \otimes_{\calR_A}
\tilde{\calR}_{\calH(\beta)}$ for each $\beta \in \calM(A')$.
This filtration has the property that each of its subquotients is a finite projective module over $\tilde{\calR}_{A'}$;
it therefore descends to a filtration of
$M \otimes_{\calR_A} \calR_{A'}$ with the analogous property thanks to
Lemma~\ref{L:descend submodule}. Each subquotient of the resulting filtration is \emph{pure}; that is, for some $c,d \in \ZZ$ with $d>0$, the action of $p^c \varphi^d$ admits an \'etale model.
\end{remark}

We now specialize to the situation considered by Berger.
\begin{theorem} \label{T:berger}
Conjecture~\ref{C:locally analytic2} holds in case $K$ is a finite unramified extension of $\QQ_p$, $K_\infty = L_\infty$ is the division field of a Lubin-Tate formal $\frako_K$-module of height
$h = [K:\QQ_p]$, and $T$ is a crystalline representation.
\end{theorem}
\begin{proof}
Set notation as in Example~\ref{exa:Berger}.
In \cite[\S 3]{berger-multi},
Berger exhibits a $(\varphi, \Gamma_K)$-equivariant map $\calR_A^{\inte} \to \tilde{\bA}^{\dagger}_{K_\infty}$.
This in turn induces a homomorphism $R \to \tilde{\bE}$ of Banach algebras,
which then gives rise to a map $\calM(\tilde{\bE}) \to \calM(R)$ of Gel'fand spectra. Since $\calM(\tilde{\bE})$
is a point, the image of this map is a point $\alpha$ of $\calM(R)$.

Berger also associates to $T$ (see \cite[Theorem~5.3]{berger-multi})
a coadmissible reflexive module $M$ over $\calR_A$ equipped with a semilinear continuous action of $\Gamma$ with the property that $M \otimes_{\calR_A} \tilde{\bB}^\dagger_{\rig}$ is \'etale.
In particular, the rank of $M$ does not jump up in a neighborhood of $\alpha$, so for some $r>0$ and some rational localization $A \to A'$ encircling $\alpha$, $M$ gives rise to a finite projective module over  $\calR_{A'}^{[r/q,r]}$. By Lemma~\ref{L:bundle to module}, $M \otimes_{\calR_A} \calR_{A'}$ is itself finite projective and hence a $(\varphi, \Gamma)$-module over $\calR_{A'}$.

Using Theorem~\ref{T:spread etale}, we may choose a rational localization $A \to A'$ such that
as a $\varphi$-module, $M \otimes_{\calR_A} \calR_{A'}$ admits an \'etale model; using the compactness of $\Gamma$,
we may further ensure that $\Gamma$ acts on $A'$. This yields the desired result.
\end{proof}

\begin{remark}
For $h = 1$, the spectrum of $\tilde{\bE}$ is the singleton set $\{\alpha\}$
and so the localization phenomenon does not arise. For $h > 1$, the point $\alpha$ is nonclassical; that is,
it does not arise from a maximal ideal of $A$. This is related to the fact that the ring $\calR_A$ does not admit a specialization map back to a power series ring in one variable  as in Example~\ref{exa:kisin-ren}, so
the failure of overconvergent descent in that example does not preclude overconvergent descent in
Example~\ref{exa:Berger}.
\end{remark}

\end{document}